\documentclass{amsart}

\usepackage{amsthm}
\usepackage{graphicx}
\usepackage{stmaryrd} 
\usepackage{tikz}
\usepackage{tikz-cd}
\usepackage{ bbold }
\usetikzlibrary{arrows}
\usepackage{verbatim}
\usepackage{amssymb,amsfonts}
\usepackage[all,arc]{xy}
\usepackage{enumerate}
\usepackage{mathrsfs, mathabx}
\usepackage{xcolor}[2007/01/21]
\usepackage{hyperref}
\usepackage{soul}
\usepackage{esint} 
\usepackage[all]{xy}
\usepackage[margin=1.0in]{geometry}
\usepackage[
   open,
   openlevel=2,
   atend
 ]{bookmark}[2011/12/02]
\usepackage{graphics}

\bookmarksetup{color=blue}

\theoremstyle{definition} 
\newtheorem{thm}{Theorem}[section]
\newtheorem{cor}[thm]{Corollary}

\newtheorem{lem}[thm]{Lemma}
\newtheorem{conj}[thm]{Conjecture}

\theoremstyle{definition}

\newtheorem{rmk}[thm]{Remark}

\theoremstyle{definition}

\theoremstyle{remark}

\newcommand{\GL}{\mathrm{GL}}

\newcommand{\Aut}{\mathrm{Aut}}

\newcommand{\bZ}{\mathbb{Z}}
\newcommand{\bQ}{\mathbb{Q}}

\newcommand{\bF}{\mathbb{F}}

\newcommand{\mr}{\mathrm}

\newcommand{\Prob}{\mathrm{Prob}}

\newcommand{\M}{\mr{M}}
\newcommand{\rk}{\mr{rank}}

\newcommand{\cok}{\mathrm{cok}}
\newcommand{\diag}{\mathrm{diag}}

\newcommand{\ra}{\rightarrow}

\newcommand{\tra}{\twoheadrightarrow}

\newcommand{\lt}{\left}
\newcommand{\rt}{\right}

\newcommand{\op}{\oplus}

\newcommand{\be}{\begin{enumerate}}
\newcommand{\ee}{\end{enumerate}}

\newcommand{\bi}{\begin{itemize}}
\newcommand{\ei}{\end{itemize}}

\newcommand{\bbm}{\begin{bmatrix}}
\newcommand{\ebm}{\end{bmatrix}}

\numberwithin{equation}{section}

\begin{document}

\title[Polynomial equations and cokernels of random matrices modulo a prime power]{Polynomial equations for matrices over integers modulo a prime power and the cokernel of a random matrix}
\date{\today}
\author{Gilyoung Cheong, Yunqi Liang, and Michael Strand}
\address{Department of Mathematics, University of California, Irvine, 340 Rowland Hall, Irvine, CA 92697.}
\email{gilyounc@uci.edu}

\begin{abstract}
Given a prime $p$ and a positive integer $k$, let $\M_{n}(\bZ/p^{k}\bZ)$ be the ring of $n \times n$ matrices over $\bZ/p^{k}\bZ$. We consider the number of solutions $X \in \M_{n}(\bZ/p^{k}\bZ)$ to the polynomial equation $P(X) = 0$, where $P(t)$ is a monic polynomial in $(\bZ/p^{k}\bZ)[t]$ whose reduction modulo $p$ is square-free over the finite field $\bF_{p}$ of $p$ elements. Noting that $P(X) = 0$ if and only if $\cok(P(X)) \simeq (\bZ/p^{k}\bZ)^{n}$, we give a conjectural generalization of counting solutions to $P(X) = 0$ as the distribution of the cokernel $\cok(P(X))$ of $P(X)$ up to isomorphisms, where $X$ is a uniform random matrix in $\M_{n}(\bZ/p^{k}\bZ)$. This distribution involves an explicit formula when we fix the residue class of $X$ modulo $p$. We prove this conjecture for the special case when the image of $P(t)$ in $\bF_{p}[t]$ modulo $p$ is irreducible. We explain how the distribution we obtain is closely related to the Cohen-Lenstra distribution. Our proof involves algebraic and combinatorial arguments in linear algebra over $\bZ/p^{k}\bZ$ and builds upon a previous work of Cheong and Kaplan. 
\end{abstract}

\maketitle

\section{Introduction}

\hspace{3mm} Let $\M_{n}(A)$ be the ring of $n \times n$ matrices over a commutative ring $A$ with unity. When $A$ has finite size, it is natural to ask how many solutions $X$ in $\M_{n}(A)$ there are for the polynomial equation $P(X) = 0$, where $P(t) \in A[t]$ is a monic polynomial. Whether $X \in \M_{n}(A)$ satisfies $P(X) = 0$ is invariant under conjugation by $\GL_{n}(A)$, so this problem comes down to classifying similarity classes in $\M_{n}(A)$ satisfying the given polynomial equation. This classification is well-known when $A$ is a field, and in this case, the number of the solutions $X \in \M_{n}(A)$ to $P(X) = 0$ was first computed by Hodges \cite{Hod}. When $A$ is a PID (not necessarily of finite size) and $P(t) \in A[t]$ is a degree $2$ polynomial that splits over $A$, Dafni and Vaserstein \cite{DV} classified all similarity classes of $X \in \M_{n}(A)$ such that $P(X) = 0$. However, this problem is open for various choices of $A$ and $P(t) \in A[t]$ since the classification of similarity classes in $\M_{n}(A)$ is open for general $A$.

\

\hspace{3mm} In this paper, we first study this problem of counting $X \in \M_{n}(A)$ satisfying $P(X) = 0$ when $A = \bZ/p^{k}\bZ$ for any $k \in \bZ_{\geq 1}$ and a monic polynomial $P(t) \in (\bZ/p^{k}\bZ)[t]$ whose image in $\bF_{p}[t]$ modulo $p$ is square-free, where $\bF_{p} = \bZ/p\bZ$, the finite field of prime size $p$. In this case, we shall see that it is easy to classify the similarity classes of $X \in \M_{n}(\bZ/p^{k}\bZ)$ satisfying $P(X) = 0$, which leads us to the answer given by Theorem \ref{eqn}. Moreover, noting that $P(X) = 0$ if and only if the cokernel $\cok(P(X))$ of $P(X)$ is isomoprhic to $(\bZ/p^{k}\bZ)^{n}$, we may interpret the problem of counting solutions to $P(X) = 0$ as a special case of counting $X \in \M_{n}(\bZ/p^{k}\bZ)$ such that $\cok(P(X)) \simeq G$ for a fixed finite abelian $p$-group $G$. We give a conjectural answer to this count, and with a fixed residue class $\bar{X} \in \M_{n}(\bF_{p})$ of $X$ modulo $p$, the conjecture involves an exact formula that does not depend on $\bar{X}$. (See Conjecture \ref{conj}.) Our main result resolves the conjecture when the image of $P(t) \in (\bZ/p^{k}\bZ)[t]$ in $\bF_{p}[t]$ modulo $p$ is irreducible. (See Theorem \ref{main}.) Resolving the full conjecture is work in progress, which builds upon the techniques from this paper. For a general monic polynomial $P(t) \in (\bZ/p^{k}\bZ)[t]$ that is not square-free modulo $p$, we do not even have a conjecture, which leaves an interesting direction for future work.

\

\hspace{3mm} We present our main result and conjecture in the form of probability for a uniform random matrix $X \in \M_{n}(\bZ/p^{k}\bZ)$ diving our count by $|\M_{n}(\bZ/p^{k}\bZ)| = p^{kn^{2}}$. If $k \in \bZ_{\geq 1}$ satisfies $p^{k-1}G = 0$, our main result computes the probability that $\cok(P(X)) \simeq G$, where $X$ is chosen from $\M_{n}(\bZ_{p})$ with the Haar probability measure, where $\bZ_{p}$ is the ring of $p$-adic integers. When we let $n \ra \infty$, this distribution weakly converges to the Cohen-Lenstra distribution of modules over $\bZ_{p}[t]/(P(t))$ with finite size, which gives a different proof of a result due to Lee \cite{Lee}. (See Theorem \ref{lee}.)

\

\hspace{3mm} For the rest of this paper, we fix a prime $p$. Before we state our main result and conjecture, we give the statement regarding the number of $X \in \M_{n}(\bZ/p^{k}\bZ)$ such that $P(X) = 0$, when $P(t) \in (\bZ/p^{k}\bZ)[t]$ is a monic polynomial whose image in $\bF_{p}[t]$ modulo $p$ is square-free. 

\

\begin{thm}\label{eqn} Let $k \in \bZ_{\geq 1}$. Consider a monic polynomial $P(t) \in (\bZ/{p}^{k}\bZ)[t]$ whose image in $\bF_{p}[t]$ modulo $p$ is square-free so that by Hensel's lemma, we have
\[P(t) = P_{1}(t) \cdots P_{l}(t)\]
for some monic polynomials $P_{1}(t), \dots, P_{l}(t) \in (\bZ/p^{k}\bZ)[t]$ whose images in $\bF_{p}[t]$ modulo $p$ are distinct and irreducible. For any $n \in \bZ_{\geq 1}$, we have
\[\#\{X \in \M_{n}(\bZ/p^{k}\bZ) : P(X) = 0\} = p^{kn^{2}} \prod_{i=1}^{n}(1 - p^{-i}) \sum_{\substack{(n_{1}, \dots, n_{l}) \in (\bZ_{\geq 0})^{n}: \\ n_{1}\deg(P_{1}) + \cdots + n_{l}\deg(P_{l}) = n}} \prod_{j=1}^{l} \frac{1}{p^{k\deg(P_{j})n_{j}^{2}}\prod_{i=1}^{n_{j}}(1 - p^{-i\deg(P_{j})})}.\]
Equivalently, dividing by $|\M_{n}(\bZ/p^{k}\bZ)| = p^{kn^{2}}$, we have
\[\underset{X \in \M_{n}(\bZ/p^{k}\bZ)}{\Prob}(P(X) = 0) = \prod_{i=1}^{n}(1 - p^{-i}) \sum_{\substack{(n_{1}, \dots, n_{l}) \in (\bZ_{\geq 0})^{n}: \\ n_{1}\deg(P_{1}) + \cdots + n_{l}\deg(P_{l}) = n}} \prod_{j=1}^{l} \frac{1}{p^{k\deg(P_{j})n_{j}^{2}}\prod_{i=1}^{n_{j}}(1 - p^{-i\deg(P_{j})})}.\]
\end{thm}

\

\begin{rmk} When $k = 1$, Theorem \ref{eqn} is due to Hodges (\cite[(2.5) and (3.4)]{Hod} or \cite[Theorem 3.1]{Hua}) for any monic polynomial $P(t) \in \bF_{p}[t]$ even without the square-free condition. For $k \geq 2$, we are unable to find a generalization of the result of Hodges unless we assume that $P(t)$ is square-free modulo $p$ due to our inability to classify similarity classes of $n \times n$ matrices over $\bZ/p^{k}\bZ$. Some general criteria to tell what matrices are similar to each other are available due to Guralnick \cite[Theorem 3.2]{Gur}, but unfortunately, they do not admit an explicit classification of the similarity classes suitable for our purpose. We also remark that the proof of Theorem \ref{eqn} is quite straightforward because it is easy to classify similarity classes of $X \in \M_{n}(\bZ/p^{k}\bZ)$ such that $P(X) = 0$ if $P(t)$ is square-free modulo $p$. Nevertheless, since we could not find this particular statement in the literature, we include a proof of Theorem \ref{eqn} in Section \ref{exp}. What we find nontrivial is a further conjectural generalization of Theorem \ref{eqn}, which we give as Conjecture \ref{conj}. When $P(t)$ is irreducible modulo $p$, this conjecture is proven as our main theorem, stated as Theorem \ref{main}. This generalizes the corresponding special case of Theorem \ref{eqn}.
\end{rmk}

\

\hspace{3mm} Note that for any $X \in \M_{n}(\bZ/p^{k}\bZ)$, we have $P(X) = 0$ if and only if the cokernel $\cok(P(X))$ of $P(X)$ is isomorphic to $(\bZ/p^{k}\bZ)^{n}$. Hence, Theorem \ref{eqn} computes the probability that a uniform random matrix $X \in \M_{n}(\bZ/p^{k}\bZ)$ satisfies $\cok(P(X)) \simeq (\bZ/p^{k}\bZ)^{n}$ as abelian groups. Having this in mind, we state our main theorem, which is a generalization of Theorem \ref{eqn} when $P(t)$ is irreducible modulo $p$. We denote by $\Aut_{R}(G)$ the group of $R$-linear automorphisms of a module $G$ over a commutative ring $R$ with unity.

\

\begin{thm}[Main theorem]\label{main} Let $k \in \bZ_{\geq 1}$ and consider a monic polynomial $P(t) \in (\bZ/p^{k}\bZ)[t]$ whose image in $\bF_{p}[t]$ modulo $p$ is irreducible. Let
\bi
	\item $R := (\bZ/p^{k}\bZ)[t]/(P(t))$,
	\item $G$ be a finite size $R$-module,
	\item $q := p^{\deg(P)}$,
	\item $\bF_{q} := \bF_{p}[t]/(P(t))$,
	\item $r_{q}(G) := \dim_{\bF_{q}}(G/pG)$, and
	\item $u_{q}(G) := \dim_{\bF_{q}}(p^{k-1}G)$.
\ei

Fix $n \in \bZ_{\geq 1}$. For any $\bar{X} \in \M_{n}(\bF_{p})$ such that $\cok(P(\bar{X})) \simeq G/pG$, we have
\[\underset{X \in \M_{n}(\bZ/p^{k}\bZ)}{\Prob}\left(\begin{array}{c}
\cok(P(X)) \simeq G \text{ and}\\
X \equiv \bar{X} \pmod{p}
\end{array}\right)
= \dfrac{p^{\deg(P)r_{q}(G)^{2} - n^{2}}}{|\Aut_{R}(G)|\prod_{i=1}^{u_{q}(G)}(1 - p^{-i\deg(P)})}\displaystyle\prod_{i=1}^{r_{q}(G)}(1 - p^{-i\deg(P)})^{2}.\]
\end{thm}

\

\begin{rmk} In Theorem \ref{main}, the condition $\cok(\bar{X}) \simeq G/pG$ is inevitable because if it is not satisfied, then the probability in Theorem \ref{main} is automatically $0$. Note that $\cok(P(X))$ is a module over $R = (\bZ/p^{k}\bZ)[t]/(P(t))$ for any $X \in \M_{n}(\bZ/p^{k}\bZ)$, where the action of $\bar{t} \in R$ is given as the left multiplication by $X$, so it is inevitable to assume that the given finite abelian $p$-group $G$ is an $R$-module. Since $p^{k}G = 0$, the action of $pR$ on $p^{k-1}G$ is trivial, so $p^{k-1}G$ is a vector space over $R/pR = \bF_{p}[t]/(P(t)) = \bF_{q}$, which is a field of $q = p^{\deg(P)}$ elements. Thus, the notation $u_{q}(G) = \dim_{\bF_{q}}(p^{k-1}G)$ makes sense. We note that $p^{k-1}G = 0$ if and only if $u_{q}(G) = 0$.
\end{rmk}

\

\hspace{3mm} We now state a conjectural generalization of Theorems \ref{eqn} and \ref{main}:

\begin{conj}\label{conj} Let $k \in \bZ_{\geq 1}$. Consider a monic polynomial $P(t) \in (\bZ/{p}^{k}\bZ)[t]$ whose image in $\bF_{p}[t]$ modulo $p$ is square-free so that by Hensel's lemma, we have
\[P(t) = P_{1}(t) \cdots P_{l}(t)\]
for some monic polynomials $P_{1}(t), \dots, P_{l}(t) \in (\bZ/p^{k}\bZ)[t]$ whose images in $\bF_{p}[t]$ modulo $p$ are distinct and irreducible.  Let
\bi
	\item $R_{j} := (\bZ/p^{k}\bZ)[t]/(P_{j}(t))$,
	\item $G_{j}$ be a finite size $R_{j}$-module,
	\item $d_{j} := \deg(P_{j})$,
	\item $q_{j} := p^{d_{j}}$,
	\item $\bF_{q_{j}} := \bF_{p}[t]/(P_{j}(t))$,
	\item $r_{q_{j}}(G_{j}) := \dim_{\bF_{q_{j}}}(G_{j}/pG_{j})$, and
	\item $u_{q_{j}}(G_{j}) := \dim_{\bF_{q_{j}}}(p^{k-1}G_{j})$.	
\ei

Fix $n \in \bZ_{\geq 1}$. For any $\bar{X} \in \M_{n}(\bF_{p})$ such that $\cok(P_{j}(\bar{X})) \simeq G_{j}/pG_{j}$, we have
\[\underset{X \in \M_{n}(\bZ/p^{k}\bZ)}{\Prob}\left(\begin{array}{c}
\cok(P_{j}(X)) \simeq G_{j} \\
\text{for } 1 \leq j \leq l \text{ and} \\
X \equiv \bar{X} \pmod{p}
\end{array}\right)
= p^{- n^{2}}\displaystyle\prod_{j=1}^{l}\dfrac{p^{d_{j} r_{q_{j}}(G_{j})^{2}}}{|\Aut_{R_{j}}(G_{j})|\prod_{i=1}^{u_{q_{j}}(G_{j})}(1 - p^{-id_{j}})}\displaystyle\prod_{i=1}^{r_{q_{j}}(G_{j})}(1 - p^{-id_{j}})^{2}.\]
\end{conj}

\

\begin{rmk} We briefly explain how Conjecture \ref{conj} generalizes Theorem \ref{eqn}. Let $G$ be a finite size module over $(\bZ/p^{k}\bZ)[t]/(P(t))$. Then it is always the case that $G \simeq G_{1} \times \cdots \times G_{l}$ as abelian groups, where each $G_{j}$ is an $R_{j}$-module. Thus, we have
\begin{align*}
\underset{X \in \M_{n}(\bZ/p^{k}\bZ)}{\Prob}(
\cok(P(X)) \simeq G) &= \sum_{\bar{X} \in \M_{n}(\bF_{p})}\underset{X \in \M_{n}(\bZ/p^{k}\bZ)}{\Prob}\left(\begin{array}{c}
\cok(P(X)) \simeq G \\
X \equiv \bar{X} \pmod{p}
\end{array}\right) \\
&= \sum_{\bar{X} \in \M_{n}(\bF_{p})}\sum_{\substack{([G_{1}], \dots, [G_{l}]): \\ 
G_{1} \times \cdots \times G_{l} \simeq G
}} \underset{X \in \M_{n}(\bZ/p^{k}\bZ)}{\Prob}\left(\begin{array}{c}
\cok(P_{j}(X)) \simeq G_{j} \\
\text{for } 1 \leq j \leq l \text{ and} \\
X \equiv \bar{X} \pmod{p}
\end{array}\right),
\end{align*}
where each $[G_{j}]$ for the sum is the isomorphism class of a finite abelian $p$-group $G_{j}$ that has an $R_{j}$-module structure. Theorem \ref{eqn} is the case when $G = (\bZ/p^{k}\bZ)^{n}$, and for this case, possibilities of $G_{j}$ that can appear in the above sum are quite restricted, which helps us deduce the conclusion of Theorem \ref{eqn}. (We give a more detailed proof in Section \ref{exp}.)
\end{rmk}

\

\subsection{Status of Conjecture \ref{conj}} The first case of Conjecture \ref{conj} is due to Friedman and Washington \cite[p.7, ``$\#_{H}(\bar{R})$'']{FW} when $l=1$ and $\deg(P) = 1$ with $p^{k-1}G = 0$ (i.e., $u_{q}(G) = 0$). The first author and Huang \cite[p.401]{CH} resolved more cases of Conjecture \ref{conj} when $G_{1} = \cdots = G_{l-1} = 0$ and $\deg(P_{l}) = 1$ with $p^{k-1}G = 0$. Later, the first author and Kaplan \cite{CK} generalized this and showed that Conjecture \ref{conj} is true as long as $\deg(P_{1}), \deg(P_{2}), \dots, \deg(P_{l}) \leq 2$ with $p^{k-1}G = 0$, and following a similar argument to the proof of Theorem \ref{main}, the result in \cite{CK} can be generalized for $G$ such that $p^{k-1}G \neq 0$ (i.e., $u_{q}(G) > 0$). Even for the special case $P(t) = t$, proving Theorem \ref{main} is quite tricky when $p^{k-1}G \neq 0$. For example, the proof from Friedman and Washington \cite[p.236]{FW} using the Haar probability measure on $\M_{n}(\bZ_{p})$, where $\bZ_{p}$ is the ring of $p$-adic integers, works only when $p^{k-1}G = 0$ but does not apply when $p^{k-1}G \neq 0$. Our proof (i.e., proof of Lemma \ref{R}) comes from a careful modification of a counting argument due to the first author and Kaplan, which can be found in the proof of \cite[Lemma 3.8]{CK}.

\

\hspace{3mm} In any of the previous results mentioned above, the authors did not consider the case $p^{k-1}G \neq 0$ of Theorem \ref{main} (or Conjecture \ref{conj}) because they were only interested in the distribution of $\cok(P(X))$, where $X$ is an $n \times n$ Haar-random matrix over $\bZ_{p}$. However, the case $p^{k-1}G \neq 0$ is interesting because counting solutions $X \in \M_{n}(\bZ/p^{k}\bZ)$ to $P(X) = 0$ corresponds to the case $G = (\bZ/p^{k}\bZ)^{n}$, where we have $p^{k-1}G \neq 0$. Building upon the techniques from this paper for the proof of Theorem \ref{main}, proving Conjecture \ref{conj} in full generality is joint work in progress by the first author, Kaplan, and Huang.

\

\subsection{Generalization of counting square matrices of specific rank over $\bF_{p}$} The following is an interesting corollary of Theorem \ref{main} by taking $P(t) = t$:

\begin{cor}\label{app1} Fix any $k \in \bZ_{\geq 1}$ and a finite abeian $p$-group $G$ such that $p^{k}G = 0$. Writing $r_{p}(G) := \dim_{\bF_{p}}(G/pG)$ and $u_{p}(G) := \dim_{\bF_{p}}(p^{k-1}G)$, we have
\[\underset{X \in \M_{n}(\bZ/p^{k}\bZ)}{\Prob}(\cok(X) \simeq G)
= \dfrac{1}{|\Aut_{\bZ}(G)|}\displaystyle\lt(\prod_{i=u_{p}(G) + 1}^{n}(1 - p^{-i})\rt) \lt(\prod_{i=n - r_{p}(G) + 1}^{n}(1 - p^{-i})\rt)\]
\end{cor}

\

\hspace{3mm} When $k = 1$ in Corollary \ref{app1}, we have $pG = 0$, so $G \simeq \bF_{p}^{r}$ for some $0 \leq r \leq n$, which implies that 
\bi
	\item $r_{p}(G) = r$, 
	\item $u_{p}(G) = r$, and
	\item $|\Aut_{\bZ}(G)| = |\GL_{r}(\bF_{p})| = p^{r^{2}}\prod_{i=1}^{r}(1 - p^{-i})$. 
\ei
In this case, for any $\bar{X} \in \M_{n}(\bF_{p})$, saying $\cok(\bar{X}) \simeq G$ is equivalent to $\rk(\bar{X}) = n - r$, so Corollary \ref{app1} generalizes the following well-known formula:

\begin{thm}[$\bF_{p}$-matrices with specific rank]\label{rk} Given any integers $0 \leq r \leq n$, we have
\[\underset{\bar{X} \in \M_{n}(\bF_{p})}{\Prob}(\rk(\bar{X}) = n - r) = \frac{p^{- r^{2}}\prod_{i=r+1}^{n}(1 - p^{-i})^{2}}{\prod_{i=1}^{n-r}(1 - p^{-i})}.\]
\end{thm}

\

\begin{rmk} Deducing Corollary \ref{app1} from Theorem \ref{main} uses Theorem \ref{rk}, so it does not provide a new proof of Theorem \ref{rk}.
\end{rmk}

\

\subsection{Distributions of random $p$-adic integral matrices and the Cohen-Lenstra distribution} Denote by $\bZ_{p}$ the ring of $p$-adic integers. Given a finite abelian $p$-group $G$, choose any $k \in \bZ_{\geq 1}$ such that $p^{k-1}G = 0$. Then Theorem \ref{main} computes the probability that $\cok(P(X)) \simeq G$, where $X$ is a random matrix in $\M_{n}(\bZ_{p})$ with respect to the Haar probability measure.

\

\begin{cor}\label{app2} Let $P(t) \in \bZ_{p}[t]$ be a monic polynomial whose image in $\bF_{p}[t]$ modulo $p$ is irreducible and $G$ a finite size module over $\bZ_{p}[t]/(P(t))$. Fix $n \in \bZ_{\geq 1}$. For any $\bar{X} \in \M_{n}(\bF_{p})$ such that $\cok(P(\bar{X})) \simeq G/pG$, we have
\[\underset{X \in \M_{n}(\bZ_{p})}{\Prob}\left(\begin{array}{c}
\cok(P(X)) \simeq G \\
\text{and } X \equiv \bar{X} \pmod{p}
\end{array}\right) = \frac{p^{\deg(P)(\dim_{\bF_{p}[t]/(P(t))}(G/pG))^{2}- n^{2}}}{|\Aut_{\bZ_{p}[t]/(P(t))}(G)|}\prod_{i=1}^{\dim_{\bF_{p}[t]/(P(t))}(G/pG)}(1 - p^{-i\deg(P)})^{2}.\]
\end{cor}

\

\hspace{3mm} When we let $n \ra \infty$, Corollary \ref{app2} implies the following result of Lee \cite{Lee}:

\begin{thm}[Lee]\label{lee} Let $P(t) \in \bZ_{p}[t]$ be a monic polynomial whose image in $\bF_{p}[t]$ modulo $p$ is irreducible and $G$ a finite size module over $\bZ_{p}[t]/(P(t))$. We have
\[\lim_{n \ra \infty}\underset{X \in \M_{n}(\bZ_{p})}{\Prob}(\cok(P(X)) \simeq G) = \frac{1}{|\Aut_{\bZ_{p}[t]/(P(t))}(G)|}\prod_{i=1}^{\infty}(1 - p^{-i\deg(P)}).\]
\end{thm}

\

\begin{rmk} Theorem \ref{main} with any $k \in \bZ_{\geq 1}$ such that $p^{k-1}G = 0$ implies Corollary \ref{app2} by definition of the Haar probability measure on $\M_{n}(\bZ_{p})$. (See \cite[Lemma 4.3]{CH} for more details.) To see how Corollary \ref{app2} implies Theorem \ref{lee}, one can use the fact \cite[Lemma 2.2]{CK} that
\[\lim_{n \ra \infty}\underset{\bar{X} \in \M_{n}(\bF_{p})}{\Prob}(\cok(P(\bar{X})) \simeq G/pG) = \frac{p^{-r^{2}\deg(P)}\prod_{i=1}^{\infty}(1 - p^{-i\deg(P)})}{\prod_{i=1}^{r}(1 - p^{-i\deg(P)})^{2}},\]
where $r =  \dim_{\bF_{p}[t]/(P(t))}(G/pG)$. We also note that Lee's result is more general than Theorem \ref{lee}, which follows from Conjecture \ref{conj}. Lee's original proof is different from this approach.
\end{rmk}

\

\begin{rmk} If $P(t) \in \bZ_{p}[t]$ is a monic polynomial whose image in $\bF_{p}[t]$ modulo $p$ is irreducible, then $R = \bZ_{p}[t]/(P(t))$ is a DVR (discrete valuation ring) with the residue field $R/pR = \bF_{p}[t]/(P(t))$ of size $q = p^{\deg(P)}$. The set of isomorphism classes of finite size $R$-modules have a discrete probability distribution (i.e., a discrete probability measure) called the \textbf{Cohen--Lenstra distribution} defined by
\[\{[G]\} \mapsto \frac{1}{|\Aut_{R}(G)|}\prod_{i=1}^{\infty}(1 - p^{-i\deg(P)})\]
for every finite size $R$-module $G$, where $[G]$ is its isomorphism class.

\

\hspace{3mm} This distribution was first studied by Cohen and Lenstra in \cite{CL} to provide heuristics to predict the distribution of the $p$-part of a random imaginary quadratic extension of $\bQ$. Such heuristics are given for $\deg(P) = 1$ so that $R = \bZ_{p}$, noting that finite size $\bZ_{p}$-modules are identical to finite abelian $p$-groups. Recently, there have been numerous activities in studying distributions that are related to the Cohen--Lenstra distribution arising from random matrices (e.g., \cite{FK}, \cite{Lee}, \cite{Lee-B}, \cite{Lee-C}, \cite{NW}, \cite{Van}, \cite{Woo17}, \cite{Woo19}) over $\bZ_{p}$ or $\bF_{p}$. Our work reveals more explicit examples via Theorem \ref{main}.
\end{rmk}

\

\subsection*{Acknowledgment} We thank Alessandra Pantano for helping us initiate the project as a summer undergraduate research opportunity supervised by the first author. We thank Yifeng Huang and Myungjun Yu for helpful conversations regarding this paper. We thank Nathan Kaplan and Alison Miller for helpful comments regarding an earlier draft of this paper.

\

\section{More on Theorem \ref{eqn} and Conjecture \ref{conj}}\label{exp}

\hspace{3mm} In this section, we explain how Conjecture \ref{conj} implies Theorem \ref{eqn}. Then we give a proof of Theorem \ref{eqn} without assuming Conjecture \ref{conj}.

\

\subsection{Conjecture \ref{conj} implies Theorem \ref{eqn}} Assume the hypotheses of Theorem \ref{eqn}. Given any $X \in \M_{n}(\bZ/p^{k}\bZ)$, we note that $\cok(P(X))$ is not just a finite abelian $p$-group, but it is also a module over $R := (\bZ/p^{k}\bZ)[t]/(P(t))$. (This observation is crucial whenever $\deg(P) \geq 2$.) By the Chinese Remainder Theorem, we have a ring isomorphism
\[R \simeq R_{1} \times \cdots \times R_{l}\]
given by $(r \mod (P)) \mapsto (r \mod (P_{1}), r \mod (P_{2}), \dots, r \mod (P_{l}))$, where $R_{j} := (\bZ/p^{k}\bZ)[t]/(P_{j}(t))$. Under this isomorphism, consider
\bi
	\item $e_{1} \in R$ that corresponds to $(1 \mod (P_{1}), 0 \mod (P_{2}), 0 \mod (P_{3}), \dots, 0 \mod (P_{l}))$,
	\item $e_{2} \in R$ that corresponds to $(0 \mod (P_{1}), 1 \mod (P_{2}), 0 \mod (P_{3}), \dots, 0 \mod (P_{l}))$,
	\item and similarly for $e_{3}, \dots, e_{l} \in R$.
\ei
Note that $1 = e_{1} + \cdots + e_{l}$ in $R$ and for any $R$-module $G$, we have
\[G = G_{1} \op G_{2} \op \cdots \op G_{l},\]
where $G_{j} = e_{j}G$, a module over $R_{j}$. We have an $R_{j}$-linear isomorphism
\[\cok(P_{j}(X)) \simeq e_{j}\cok(P(X))\]
given by $(v \mod (P_{j}(X))) \mapsto e_{j}(v \mod (P(X)))$, where $v \in (\bZ/p^{k}\bZ)^{n}$, so as abelian groups, we have
\[\cok(P(X)) \simeq \cok(P_{1}(X)) \times \cdots \times \cok(P_{l}(X))\]
given by $(v \mod (P(X))) \mapsto (v \mod (P_{1}(X)), \dots, v \mod (P_{l}(X)))$.

\

\hspace{3mm} Any $R_{j}$-module is also a module over $\bZ_{p}[t]/(P_{j}(t))$ by extension of scalars. Since $\bZ_{p}[t]/(P_{j}(t))$ is a PID (as it is a DVR), this helps us classify finite abelian $p$-groups that have $R_{j}$-module structures: they are precisely the ones of the form $H^{\deg(P_{j})}$ for some finite abelian $p$-group $H$, up to isomorphisms. Moreover, it also follows that any two finite size $R_{j}$-modules are isomorphic as $R_{j}$-modules if and only they are isomorphic as abelian groups. We now explain how Conjecture \ref{conj} implies Theorem \ref{eqn}:

\

\begin{proof}[Proof that Conjecture \ref{conj} implies Theorem \ref{eqn}] Any $X \in \M_{n}(\bZ/p^{k}\bZ)$ such that $P(X) = 0$ must satisfy
\[\cok(P_{1}(X)) \times \cdots \times \cok(P_{l}(X)) \simeq \cok(P(X)) \simeq (\bZ/p^{k}\bZ)^{n}\]
as abelian groups, so for $1 \leq j \leq l$, we must have
\[\cok(P_{j}(X)) \simeq (\bZ/p^{k}\bZ)^{n_{j} \deg(P_{j})}\]
for some $n_{j} \in \bZ_{\geq 0}$ such that $n_{1}\deg(P_{1}) + \cdots + n_{l}\deg(P_{l}) = n$. If $\bar{X} \in \M_{n}(\bF_{p})$ denotes the residue of $X$ modulo $p$, then it follows that $\dim_{\bF_{p}}(\cok(P_{j}(\bar{X})) = n_{j}\deg(P_{j})$. Thus, for any fixed $\bar{X} \in \M_{n}(\bF_{p})$, we have
\begin{align*}
\underset{X \in \M_{n}(\bZ/p^{k}\bZ)}{\Prob}\left(\begin{array}{c}
P(X) = 0 \text{ and} \\
X \equiv \bar{X} \pmod{p}
\end{array}\right) &= \underset{X \in \M_{n}(\bZ/p^{k}\bZ)}{\Prob}\left(\begin{array}{c}
\cok(P(X)) \simeq (\bZ/p^{k}\bZ)^{n} \\
\text{and } X \equiv \bar{X} \pmod{p}
\end{array}\right) \\
&= \underset{X \in \M_{n}(\bZ/p^{k}\bZ)}{\Prob}\left(\begin{array}{c}
\cok(P_{j}(X)) \simeq (\bZ/p^{k}\bZ)^{n_{j}\deg(P_{j})} \\
\text{for } 1 \leq j \leq l \text{ and} \\
X \equiv \bar{X} \pmod{p}
\end{array}\right),
\end{align*}
where $n_{j} \in \bZ_{\geq 0}$ are chosen so that $\dim_{\bF_{p}}(\cok(P_{j}(\bar{X}))) = n_{j}\deg(P_{j})$. Noting that
\[\Aut_{R_{j}}((\bZ/p^{k}\bZ)^{n_{j}\deg(P_{j})}) = \Aut_{R_{j}}(R_{j}^{n_{j}}) = |\GL_{n_{j}}(R_{j})| = p^{k\deg(P_{j})n_{j}^{2}}\prod_{i=1}^{n_{j}}(1 - p^{-i\deg(P_{j})}),\]
applying Conjecture \ref{conj} with
\bi
	\item $G_{j} = R_{j}^{n_{j}} = (\bZ/p^{k}\bZ)^{n_{j}\deg(P)}$ so that
	\item $r_{q_{j}}(G_{j}) = n_{j} = u_{q_{j}}(G_{j})$
\ei
gives us
\begin{align*}
\underset{X \in \M_{n}(\bZ/p^{k}\bZ)}{\Prob}\left(\begin{array}{c}
P(X) = 0 \text{ and} \\
X \equiv \bar{X} \pmod{p}
\end{array}\right) &= p^{- n^{2}}\displaystyle\prod_{j=1}^{l}\dfrac{p^{\deg(P_{j})n_{j}^{2}}}{|\Aut_{R_{j}}(G_{j})|\prod_{i=1}^{n_{j}}(1 - p^{-i\deg(P_{j})})}\displaystyle\prod_{i=1}^{n_{j}}(1 - p^{-i\deg(P_{j})})^{2}\\
&= p^{- n^{2}}\displaystyle\prod_{j=1}^{l}\dfrac{1}{ p^{(k-1) \deg(P_{j}) n_{j}^{2}}},
\end{align*}
where $n_{j} \in \bZ_{\geq 0}$ are given as above. Summing over all $\bar{X} \in \M_{n}(\bF_{p})$ such that $P(\bar{X}) = 0$ (i.e., $\cok(P(\bar{X})) = \bF_{p}^{n}$), we have
\[\underset{X \in \M_{n}(\bZ/p^{k}\bZ)}{\Prob}(P(X) = 0) = \sum_{\substack{(n_{1}, \dots, n_{l}) \in \bZ_{\geq 0}^{n} : \\
n_{1}\deg(P_{1}) + \cdots + n_{l}\deg(P_{l}) = n
}} \underset{\bar{X} \in \M_{n}(\bF_{p})}{\Prob}\left(\begin{array}{c}
\dim_{\bF_{p}}(\cok(P_{j}(\bar{X}))) = n_{j}\deg(P_{j}) \\
\text{for } 1 \leq j \leq l
\end{array}\right)\displaystyle\prod_{j=1}^{l}\dfrac{1}{ p^{(k-1) \deg(P_{j}) n_{j}^{2}}}.\]
Finally, note that
\begin{align*}
\#\left\{\begin{array}{c}
\bar{X} \in \M_{n}(\bF_{p}) : \\
\dim_{\bF_{p}}(\cok(P_{j}(\bar{X}))) = n_{j}\deg(P_{j}) \\
\text{for } 1 \leq j \leq l
\end{array}\right\} &= \frac{|\GL_{n}(\bF_{p})|}{|\GL_{n_{1}}(\bF_{p^{\deg(P_{1})}})| \cdots |\GL_{n_{l}}(\bF_{p^{\deg(P_{l})}})|} \\
&= p^{n^{2}}\prod_{i=1}^{n}(1 - p^{-i}) \prod_{j=1}^{l}\frac{1}{p^{\deg(P_{j})n_{j}^{2}}\prod_{i=1}^{n_{j}}(1 - p^{-i\deg(P_{j})})}
\end{align*}
because the number of $\bar{X} \in \M_{n}(\bF_{p})$ with the given conditions is precisely the number of ways to decompose $\bF_{p}^{n}$ into a direct sum $V_{1} \op \cdots \op V_{l}$, where each $V_{j}$ is a vector space over $\bF_{p}[t]/(P_{j}(t))$ with dimension $n_{j}$. Hence, the conclusion of Theorem \ref{eqn} follows, which finishes the proof.
\end{proof}

\

\subsection{Proof of Theorem \ref{eqn}} Next, we give a proof of Theorem \ref{eqn} without assuming Conjecture \ref{conj}. First, note that for any finite size commutative ring $A$ with unity and a monic polynomial $P(t) \in A[t]$, we have
\begin{equation}\label{sum}
\#\left\{\begin{array}{c}
X \in \M_{n}(A) : \\
P(X) = 0
\end{array}\right\} = \sum_{\substack{[M] : M \text{ module over } A[t]/(P(t)), \\
M \simeq A^{n} \text{ as modules over }A
}}\frac{|\GL_{n}(A)|}{|\Aut_{A[t]/(P(t))}(M)|}
\end{equation}
by applying the orbit-stabilizer theorem on the conjugation action of $\GL_{n}(A)$ on $\M_{n}(A)$. In the following proof, we shall see that it is easy to enumerate $[M]$ with the hypotheses given for Theorem \ref{eqn}.

\begin{proof}[Proof of Theorem \ref{eqn}] We take $A = \bZ/p^{k}\bZ$ for the sum (\ref{sum}). We have $P(t) = P_{1}(t) \cdots P_{l}(t)$, where $P_{j}(t) \in (\bZ/p^{k}\bZ)[t]$ are monic polynomials that are distinct and irreducible modulo $p$. Let  $R := (\bZ/p^{k}\bZ)[t]/(P(t))$. (Recall also that $R_{j} = (\bZ/p^{k}\bZ)[t]/(P_{j}(t))$.) As explained in the previous subsection, all the possible ways that $(\bZ/p^{k}\bZ)^{n}$ gets an $R$-module structure (up to $R$-linear isomorphism) is via a decomposition of the form
\[(\bZ/p^{k}\bZ)^{n} \simeq (\bZ/p^{k}\bZ)^{n_{1}\deg(P_{1})} \times \cdots \times (\bZ/p^{k}\bZ)^{n_{l}\deg(P_{l})},\]
where $(n_{1}, \dots, n_{l}) \in (\bZ_{\geq 0})^{l}$ satisfies $n_{1}\deg(P_{1}) + \cdots + n_{l}\deg(P_{l}) = n$ and we view
\[(\bZ/p^{k}\bZ)^{n_{j}\deg(P_{j})} = R_{j}^{n_{j}},\]
a free $R_{j}$-module. Thus, we have
\begin{align*}
\#\left\{\begin{array}{c}
X \in \M_{n}(\bZ/p^{k}\bZ) : \\
P(X) = 0
\end{array}\right\} &= |\GL_{n}(\bZ/p^{k}\bZ)|\sum_{\substack{(n_{1}, \dots, n_{l}) \in (\bZ_{\geq 0})^{l} : \\
n_{1}\deg(P_{1}) + \cdots + n_{l}\deg(P_{l}) = n
}}\prod_{j=1}^{l}\frac{1}{|\Aut_{R_{j}}(R_{j}^{n_{j}})|} \\
&= |\GL_{n}(\bZ/p^{k}\bZ)|\sum_{\substack{(n_{1}, \dots, n_{l}) \in (\bZ_{\geq 0})^{l} : \\
n_{1}\deg(P_{1}) + \cdots + n_{l}\deg(P_{l}) = n
}}\prod_{j=1}^{l}\frac{1}{|\GL_{n_{j}}(R_{j})|} \\
&= p^{kn^{2}}\prod_{i=1}^{n}(1 - p^{-i})\sum_{\substack{(n_{1}, \dots, n_{l}) \in (\bZ_{\geq 0})^{l} : \\
n_{1}\deg(P_{1}) + \cdots + n_{l}\deg(P_{l}) = n
}}\prod_{j=1}^{l}\frac{1}{p^{k\deg(P_{j})n_{j}^{2}}\prod_{i=1}^{n_{j}}(1 - p^{-i\deg(P_{j})})},
\end{align*}
as desired.
\end{proof}

\

\section{Proof of Theorem \ref{main}}

\hspace{3mm} In this section, we prove Theorem \ref{main}. Before we get to the proof, we briefly summarize the strategy:

\subsection{Strategy for proof of Theorem \ref{main}} We utilize the module structure of $G$ over $R = (\bZ/p^{k}\bZ)[t]/(P(t))$. It turns out (Lemma \ref{Lee}) that
\[\cok(P(X)) \simeq \cok_{R}(X - \bar{t}I_{n}),\]
as $R$-modules, where the cokernel on the right-hand side is taken over $R$ and $\bar{t}$ is the image of $t$ in $R$. (As usual, we denote by $I_{n}$ the $n \times n$ identity matrix.) That is, we avoid having to consider the polynomial push-forward $P(X)$ of $X$, but the price we pay is that the resulting ring $R = (\bZ/p^{k}\bZ)[t]/(P(t))$ we work over gets more complicated than the ring $\bZ/p^{k}\bZ$ we started with. Nevertheless, as in \cite{CK}, the random matrices over $R$ turns out to have a quite ``rigid'' behavior even if we consider each summand of
\[R = \bZ/p^{k}\bZ \op \bar{t} \bZ/p^{k}\bZ \op \cdots \op \bar{t}^{d-1}\bZ/p^{k}\bZ\]
separately, where $d = \deg(P)$. This behavior, which we formulate as Lemma \ref{final}, is what we ultimately need to establish to deduce Theorem \ref{main}. In \cite{CK}, it is assumed that $\deg(P) = d \leq 2$, and this assumption admitted an easy proof of Lemma \ref{final}. However, it turns out that for $d \geq 3$, this lemma is quite nontrivial, and this is the gist of our work. (See Remark \ref{special} for further explanations.)

\

\hspace{3mm} From now on, we assume the notation in the statement of Theorem \ref{main}. We have $k \in \bZ_{\geq 1}$ such that $p^{k}G = 0$. We fix the following notation. (This is as in the statement of Theorem \ref{main} except $d$.)
\bi
	\item $d := \deg(P)$;
	\item $R := (\bZ/p^{k}\bZ)[t]/(P(t))$;
	\item $\bF_{q} := R/pR \simeq \bF_{p}[t]/(P(t))$, a finite field of $q := p^{d}$ elements (because $P(t)$ is irreducible modulo $p$);
	\item $r_{q}(G) := \dim_{\bF_{q}}(G/pG)$;
	\item $u_{q}(G) := \dim_{\bF_{q}}(p^{k-1}G)$.
\ei

Theorem \ref{main} is equivalent to
\begin{equation}\label{goal}
\#\left\{\begin{array}{c}
X \in \M_{n}(\bZ/p^{k}\bZ) : \\
\cok(P(X)) \simeq G \\
\text{and } X \equiv \bar{X} \pmod{p}
\end{array}\right\}
= p^{(k-1)n^{2}}\dfrac{q^{r_{q}(G)^{2}}\prod_{i=1}^{r_{q}(G)}(1 - q^{-i})^{2}}{|\Aut_{R}(G)|\prod_{i=1}^{u_{q}(G)}(1 - q^{-i})}. 
\end{equation}
We focus on showing this for the rest of this paper.

\

\subsection{Useful lemmas} The following lemma is due to Lee, which he used in \cite{Lee}. Although the proof, which can be found in \cite[Lemma 3.2]{CK}, is quite simple, the role that this lemma plays in proving our proof of Theorem \ref{main} is extremely critical.

\begin{lem}[Lee]\label{Lee} Consider the direct sum 
\[R = (\bZ/p^{k}\bZ)[t]/(P(t)) = \bZ/p^{k}\bZ \op \bar{t} (\bZ/p^{k}\bZ) \op \cdots \op \bar{t}^{d-1} (\bZ/p^{k}\bZ).\]

Fix any $X \in \M_{n}(\bZ/p^{k}\bZ)$. The map
\[\psi \colon \frac{(\bZ/p^{k}\bZ)^{n}}{P(X)(\bZ/p^{k}\bZ)^{n}} = \cok(P(X)) \ra \cok_{R}(X - \bar{t}I_{n}) := \frac{R^{n}}{(X - \bar{t}I_{n})R^{n}}\]

defined by $\psi([v]) = [v]$, where $v \in (\bZ/p^{k}\bZ)^{n}$, is an $R$-linear isomorphism.
\end{lem}

\

\hspace{3mm} The following is a well-known formula for counting the number of $n \times n$ matrices with specific rank over $\bF_{q}$. (Theorem \ref{rk} with $r$ and $n-r$ reversed is a special case when $q$ is a prime.)

\begin{lem}\label{rk2} Fix any integers $n \geq 1$ and $0 \leq r \leq n$. The number of rank $r$ matrices in $\M_{n}(\bF_{q})$ is
\[q^{n^{2} - (n-r)^{2}}\frac{\prod_{i=n-r+1}^{n}(1 - q^{-i})^{2}}{\prod_{i=1}^{r}(1 - q^{-i})}.\]
\end{lem}

\

\hspace{3mm} For our proof, we also need an explicit formula for $|\Aut_{R}(G)|$ for the given $R$-module $G$ of finite size. Before we state the formula, note that by extension of scalars under the projection $\bZ_{p}[t]/(P(t)) \tra (\bZ/p^{k}\bZ)[t]/(P(t)) = R$ modulo $p^{k}$, we can view $G$ as a module over $\bZ_{p}[t]/(P(t))$, which is a PID. Thus, using the structure theorem for finitely generated modules over this PID, we can deduce that
\[G \simeq (R/p^{e_{1}}R)^{r_{1}} \times \cdots \times (R/p^{e_{s}}R)^{r_{s}}\]
for some integers $k \geq e_{1} \geq \cdots \geq e_{s} \geq 1$ and $r_{1}, \dots, r_{s} \geq 1$. (The case when $G$ is trivial corresponds to $s = 0$.) The formula for $|\Aut_{R}(G)|$ is given in terms of $s, e_{1}, \dots, e_{s}, r_{1}, \dots, r_{s},$ and $q = |R/pR|$.

\begin{lem}\label{Aut} We have
\[|\Aut_{R}(G)| = \prod_{i=1}^{s}q^{-r_{i}^{2}}|\GL_{r_{i}}(\bF_{q})|\lt(\prod_{1 \leq i, j \leq s}q^{\min(e_{i},e_{j})r_{i}r_{j}}\rt).\]
\end{lem}

\

\begin{rmk} The proof of Lemma \ref{Aut} is well-known. One place that contains a full proof is \cite[p.236]{FW}. (Their statement is only given for prime $q$, but a similar argument admits a proof of Lemma \ref{Aut}.)
\end{rmk}

\

\hspace{3mm} The following Lemma is obtained by modifying the statement and proof of \cite[Lemma 3.1]{CK}. This plays an important role in proving Theorem \ref{main}.

\begin{lem}\label{R} For any $\bar{Z} \in \M_{n}(\bF_{q})$ satisfying $\cok(\bar{Z}) \simeq G/pG$, we have
\[
\#\left\{\begin{array}{c}
Z \in \M_{n}(R) : \\
\cok_{R}(Z) \simeq G \\
\text{and } Z \equiv \bar{Z} \pmod{p}
\end{array}\right\} = q^{(k-1)n^{2}}\dfrac{q^{r_{q}(G)^{2}}\prod_{i=1}^{r_{q}(G)}(1 - q^{-i})^{2}}{|\Aut_{R}(G)|\prod_{i=1}^{u_{q}(G)}(1 - q^{-i})}.
\]
\end{lem}

\begin{proof} We follow the proof of \cite[Lemma 3.1]{CK}, which only covers the case $p^{k-1}G = 0$, but also add some subtle modifications to cover the case $p^{k-1}G \neq 0$. We write $r = r_{q}(G) = \dim_{\bF_{q}}(G/pG)$ for ease of notation. Having $\cok(\bar{Z}) \simeq G/pG$ is equivalent to $\rk(\bar{Z}) = n - r$ over $\bF_{q}$. Hence, we have $\bar{Q}_{1}, \bar{Q}_{2} \in \GL_{n}(\bF_{q})$ such that
\[\bar{Q}_{1}\bar{Z}\bar{Q}_{2} = \diag(0, 0, \dots, 0, 1, 1, \dots, 1),\]
the $n \times n$ diagonal matrix whose first $r$ entries are $0$ and the remaining $n-r$ entries are $1$. We fix lifts $Q_{1}, Q_{2} \in \M_{n}(R)$ of $\bar{Q}_{1}, \bar{Q}_{2} \in \M_{n}(\bF_{q}) = \M_{n}(R/pR)$, respectively, meaning $Q_{i} \equiv \bar{Q}_{i} \pmod{p}$. Let $Z \in \M_{n}(R)$ be any lift of $\bar{Z} \in \M_{n}(\bF_{q})$. Since $Q_{1}ZQ_{2}$ is a lift of $\bar{Q}_{1}\bar{Z}\bar{Q}_{2}$ that corresponds to $Z$ with $\cok_{R}(Q_{1}ZQ) \simeq \cok_{R}(Z)$, this process lets us reduce the problem to the following special case:
\[\bar{Z} = \diag(0, 0, \dots, 0, 1, 1, \dots, 1).\]
Hence, we assume this from now on.

\

\hspace{3mm} Noting that $R/pR = \bF_{q}$, we have
\[Z = \begin{bmatrix}
pA & pB \\
pC & I_{r} + pD
\end{bmatrix},\]
where
\bi
	\item $A \in \M_{n-r}(R)$ and $D \in \M_{r}(R)$;
	\item $B$ is an $r \times (n-r)$ matrix whose entries are in $R$;
	\item $C$ is an $(n-r) \times r$ matrix whose entries are in $R$.
\ei
Note that $Z$ only depend on $pA, pB, pC, pD$ not necessarily $A, B, C, D$. By elementary row and column operations, we can deduce that 
\[QZQ' = \begin{bmatrix}
pA - p^{2}B(I_{r} + pD)^{-1}C & 0 \\
0 & I_{r} + pD
\end{bmatrix}\]
for some $Q, Q' \in \GL_{n}(R)$. This implies that
\[\cok_{R}(Z) \simeq \cok_{R}(QZQ') \simeq \cok_{R}(pA - p^{2}B(I_{r} + pD)^{-1}C).\]
Note that for any fixed $pB, pC, pD$, the map $pA \mapsto pA - p^{2}B(I_{r} + pD)^{-1}C$ is a bijection from $p\M_{r}(R)$ to itself. There are $q^{(k-1)(n^{2} - r^{2})}$ ways to choose $(pB, pC, pD)$, so this implies that
\[\#\left\{\begin{array}{c}
Z \in \M_{n}(R) : \\
\cok_{R}(Z) \simeq G \\
\text{and } Z \equiv \bar{Z} \pmod{p}
\end{array}\right\} = q^{(k-1)(n^{2} - r^{2})}\#\{pA \in p\M_{r}(R) : \cok_{R}(pA) \simeq G\}.\]
It remains to show that
\begin{equation}\label{show}
\#\{pA \in p\M_{r}(R) : \cok_{R}(pA) \simeq G\} = \dfrac{q^{kr^{2}}\prod_{i=1}^{r}(1 - q^{-i})^{2}}{|\Aut_{R}(G)|\prod_{i=1}^{u_{q}(G)}(1 - q^{-i})}.
\end{equation}
To do so, we fix an arbitrary $pA \in p\M_{r}(R)$, which we can write as
\[pA = pA_{1} + p^{2}A_{2} + \cdots + p^{k-1}A_{k-1},\]
where each $A_{i}$ is an $r \times r$ matrix with entries in the set $\{a_{0} + a_{1}\bar{t} + \cdots + a_{d-1}\bar{t}^{d-1} : a_{j} \in \{0, 1, \dots, p-1\}\}$ of $q = p^{d}$ elements, which we may identify as $\bF_{q}$. (Hence, we may consider $A_{1}, A_{2}, \dots, A_{k-1}$ as matrices over $\bF_{q}$; in particular, we may speak about their ranks.) We then carefully consider what constrains ought to be imposed on each $p^{i}A_{i}$ when we require $\cok_{R}(pA) \simeq G$.

\

\hspace{3mm} To do this, we recall that we have a concrete description of $G$ (from the paragraph before Lemma \ref{Aut}) as follows:
\begin{align*}
G &\simeq (R/p^{e_{1}}R)^{r_{1}} \times \cdots \times (R/p^{e_{s}}R)^{r_{s}} \\
&\simeq \lt(\frac{(\bZ/p^{e_{1}}\bZ)[t]}{(P(t))}\rt)^{r_{1}} \times \cdots \times \lt(\frac{(\bZ/p^{e_{s}}\bZ)[t]}{(P(t))}\rt)^{r_{s}}
\end{align*}
as $R$-modules, for some integers $k \geq e_{1} > \cdots > e_{s} \geq 1$ and $r_{1}, \dots, r_{s} \geq 1$. Note that 
\bi
	\item $r_{1} + \cdots + r_{s} = r$; 
	\item the case $r=0$ corresponds to $s=0$;
	\item $p^{k-1}G = 0$ if and only if $e_{1} < k$ (which also implies $u_{q}(G) = 0$);
	\item $p^{k-1}G \neq 0$ and ony if $e_{1} = k$ (which also implies $u_{q}(G) = r_{1}$).
\ei
When $r = 0$ (so that $s = 0$), the result is trivial. Thus, assume that $r \geq 1$. We now count the number possibilities for $A_{1}, \dots, A_{k-1}$ (as matrices in $\M_{r}(\bF_{q})$) so that $\cok(pA) \simeq G$. We do this in several steps as follows.

\

\textbf{Step 1}. To satisfy the isomprhism $\cok(pA) \simeq G$, we must have $A_{1} = A_{2} = \cdots = A_{e_{s}-1} = 0$ because otherwise $\cok(pA)$ can have a copy of $R/p^{e}R$ for some $1 \leq e < e_{s}$ in its product (or direct sum) decomposition, while $G$ does not have the same. (If $e_{1} = 1$, this sentence gives a vacuous claim.) If $e_{s} = e_{1} = k$, then we have already counted all the possibilities for $A_{1}, A_{2}, \dots, A_{k-1}$ (i.e., they are all zero matrices). In this case, we also have $G \simeq (R/p^{k}R)^{r}$ with $r = r_{1} = r_{s}$, so $|\Aut_{R}(G)| = q^{kr^{2}}\prod_{i=1}^{r}(1 - q^{-i})$ and $r = r_{q}(G) = u_{q}(G)$. This implies that both sides of (\ref{show}) are $1$, so we would be done with the proof. Thus, assume that $e_{s} \neq e_{1}$ or $e_{1} < k$. Then the requirement that the copy of $R/p^{e_{s}}R$ must appear $r_{s}$ times can be equivalently stated as $\rk(A_{e_{s}}) = r_{s}$ (as a matrix over $\bF_{q}$). Suppose that we chose such $A_{e_{s}}$. So far, we have counted all the valid choices for $A_{1}, \dots, A_{e_{s}}$.

\

\textbf{Step 2}. Multiplying matrices in $\GL_{n}(\bF_{q})$ left and right to $pA$, we may assume that
\[A_{e_{s}} = \begin{bmatrix}
0 & 0 \\
0 & I_{r_{s}}
\end{bmatrix},\]
where the top left zero means the $(r - r_{s}) \times (r - r_{s})$ zero matrix. Then we write
\[p^{e_{s}+1}A_{e_{s}+1} + p^{e_{s}+2}A_{e_{s}+2} + \cdots + p^{k-1}A_{k-1} = \begin{bmatrix}
p^{e_{s}+1}B_{e_{s}+1} + p^{e_{s}+2}B_{e_{s}+2} + \cdots + p^{k-1}B_{k-1} & p^{e_{s}+1}C \\
p^{e_{s}+1}D & p^{e_{s}+1}E
\end{bmatrix},\]
where each $B_{i}$ is an $(r-r_{s}) \times (r-r_{s})$ matrix over $\bF_{q}$. Note that to guarantee $\cok(pA) \simeq G$, we must satisfy $B_{e_{s}+1} = B_{e_{s}+2} = \cdots = B_{e_{s-1}-1} = 0$ to ensure that the product decomposition of $\cok(pA)$ does not have any copies of $R/p^{e}R$ with $e_{s}< e < e_{s-1}$. We can freely choose $(p^{e_{s}+1}C, p^{e_{s}+1}D, p^{e_{s}+1}E)$, where $C, D, E$ are matrices of appropriate sizes over $\bF_{q}$. The number of choices for $(p^{e_{s}+1}C, p^{e_{s}+1}D, p^{e_{s}+1}E)$ is 
\[q^{(k-1 - e_{s})(r^{2} - (r - r_{s})^{2})} = q^{(k-1-e_{s})(r_{s}^{2} + 2r_{s}\sum_{1 \leq i < s}r_{i})}.\]
If $e_{s-1} = e_{1} = k$, then we have counted all the possibilities for $A_{1}, A_{2}, \dots, A_{k-1}$, and we may jump to Case 2 of Step 5. Otherwise, we have $e_{s-1} < e_{1}$ or $e_{1} < k$. Then we must have $\rk(B_{e_{s-1}}) = r_{s-1}$ to ensure that the copy of $R/p^{e_{s-1}}R$ appears $r_{s-1}$ times in the product decomposition of $\cok(pA)$. So far, we have counted valid choices for $A_{1}, \dots, A_{e_{s}}, \dots, A_{e_{s-1}}$ and the entries of $A_{e_{s-1}+1}, A_{e_{s-1}+2}, \dots, A_{k-1}$ except the top left $(r - r_{s}) \times (r - r_{s})$ submatrices. (This amounts to not counting $B_{e_{s-1}+1}, \dots, B_{k-1}$ yet, which is to be done in later steps.)

\ 

\textbf{Step 3}. Multiplying matrices in $\GL_{n}(\bF_{q})$ left and right to $pA$, we may assume that
\[A_{e_{s-1}} = \begin{bmatrix}
0 & 0 & \ast \\
0 & I_{r_{s-1}} & \ast \\
\ast & \ast & \ast
\end{bmatrix},\]
where the top left zero means the $(r - r_{s} - r_{s-1}) \times (r - r_{s} - r_{s-1})$ zero matrix. This may change $A_{e_{s}}, A_{e_{s}+1}, \dots, A_{e_{s-1}+1}$, but this does not affect our counting, and the submatrices with the mark $\ast$ are the ones that are already counted in the previous step. Then we may write
\[p^{e_{s-1}+1}A_{e_{s-1}+1} + p^{e_{s-1}+2}A_{e_{s-1}+2} + \cdots + p^{k-1}A_{k-1} = \begin{bmatrix}
p^{e_{s-1}+1}B'_{e_{s-1}+1} + p^{e_{s-1}+2}B'_{e_{s-1}+2} + \cdots + p^{k-1}B'_{k-1} & p^{e_{s-1}+1}C' & \ast \\
p^{e_{s-1}+1}D' & p^{e_{s-1}+1}E' & \ast \\
\ast & \ast & \ast
\end{bmatrix},\]
where each $B'_{i}$ is an $(r - r_{s} - r_{s-1}) \times (r - r_{s} - r_{s-1})$ matrix and $E'$ is an $r_{s-1} \times r_{s-1}$ matrix over $\bF_{q}$. Note that to guarantee $\cok(pA) \simeq G$, we must satisfy $B'_{e_{s-1}+1} = B'_{e_{s-1}+2} = \cdots = B'_{e_{s-2}-1} = 0$ to ensure that the product decomposition of $\cok(pA)$ does not have any copies of the form $R/p^{e}R$ with $e_{s-1}< e < e_{s-2}$. The number of free choices for $(p^{e_{s-1}+1}C', p^{e_{s-1}+1}D', p^{e_{s-1}+1}E')$ is 
\begin{align*}
q^{(k-1 - e_{s-1})((r - r_{s})^{2} - (r - r_{s} - r_{s-1})^{2})} &=q^{(k-1 - e_{s-1}) (r_{s-1}(2(r-r_{s}) - r_{s-1}))} \\
&= q^{(k-1 - e_{s-1})(r_{s-1}(r_{s-1} + 2(r_{1} + \cdots + r_{s-2})))} \\
&= q^{(k-1 - e_{s-1})(r_{s-1}^{2} + 2r_{s-1}\sum_{1 \leq i < s-1}r_{i})}.
\end{align*}
If $e_{s-2} = e_{1} = k$, then we have counted all the possibilities for $A_{1}, A_{2}, \dots, A_{k-1}$, and we may jump to Case 2 of Step 5. Otherwise, we have $e_{s-2} < e_{1}$ or $e_{1} < k$. Then we must have $\rk(B'_{e_{s-2}}) = r_{s-2}$ to ensure that the copy of $R/p^{e_{s-2}}R$ appears $r_{s-2}$ times in the product decomposition of $\cok(pA)$. So far, we have counted valid choices for $A_{1}, \dots, A_{e_{s}}, \dots, A_{e_{s-1}}, \dots, A_{e_{s-2}}$ and the entries of $A_{e_{s-2}+1}, A_{e_{s-2}+2}, \dots, A_{k-1}$ except the top left $(r - r_{s} - r_{s-1}) \times (r - r_{s} - r_{s-1})$ submatrices. (This amounts to not counting $B'_{e_{s-2}+1}, \dots, B'_{k-1}$ yet, which is to be done in the next step.)

\

\textbf{Step 5}. Continuing this way, we can count all the valid choices for $A_{1}, \dots, A_{e_{s}}, \dots, A_{e_{s-1}}, \dots, A_{e_{2}}$ and the entries of $A_{e_{1}}, A_{e_{1}+1}, \dots, A_{k-1}$ except the top left $(r_{1}+r_{2}) \times (r_{1}+r_{2})$ submatrices. We may multiply matrices in $\GL_{n}(\bF_{q})$ left and right of $pA$ so that we may assume
\[A_{e_{2}} = \begin{bmatrix}
0 & 0 & \ast \\
0 & I_{r_{2}} & \ast \\
\ast & \ast & \ast
\end{bmatrix},\]
where the top left zero means the $r_{1} \times r_{1}$ zero matrix. (This may change $A_{e_{s}}, A_{e_{s}+1}, \dots, A_{e_{2}+1}$, but this does not affect our counting.) Then we may write
\[p^{e_{2}+1}A_{e_{2}+1} + p^{e_{2}+2}A_{e_{2}+2} + \cdots + p^{k-1}A_{k-1} = \begin{bmatrix}
p^{e_{2}+1}B''_{e_{2}+1} + p^{e_{2}+2}B''_{e_{2}+2} + \cdots + p^{k-1}B''_{k-1} & p^{e_{2}+1}C'' & \ast \\
p^{e_{2}+1}D'' & p^{e_{2}+1}E'' & \ast \\
\ast & \ast & \ast
\end{bmatrix},\]
where the submatrices with the mark $\ast$ are the ones that are already counted in the previous step.

\

\hspace{3mm} There are two cases from here:

\

\textbf{Case 1}. The first case is when $e_{1} < k$ (i.e., $p^{k-1}G = 0$). Note that to guarantee $\cok(pA) \simeq G$, we must satisfy the following criteria:

\bi
	\item We must have $B''_{e_{2}+1} = B''_{e_{2}+2} = \cdots = B''_{e_{1}-1} = 0$ to ensure that the product decomposition of $\cok(pA)$ does not have any copies of the form $R/p^{e}R$ with $e_{2}< e < e_{1}$.
	\item We must have $\rk(B''_{e_{1}}) = r_{1}$ to ensure that the copy of $R/p^{e_{1}}R$ appears $r_{1}$ times in the product decomposition of $\cok(pA)$.
\ei
So far, we have counted valid choices for $A_{1}, \dots, A_{e_{s}}, \dots, A_{e_{2}}, \dots, A_{e_{1}}$ and the entries of $A_{e_{1}+1}, A_{e_{1}+2}, \dots, A_{k-1}$ except the top left $r_{1} \times r_{1}$ submatrices. Now, we note that all the possible remaining entries of $A_{e_{1}+1}, A_{e_{1}+2}, \dots, A_{k-1}$ are valid, and there are $q^{(k-1-e_{1})r_{1}^{2}}$ such choices.

\

\hspace{3mm} Following all the steps above, we have
\begin{align*}
\#\{pA \in p\M_{r}(R) : \cok_{R}(pA) \simeq G\} &= \lt(\prod_{i=0}^{s-1}\#\left\{\begin{array}{c}
X \in \M_{r-(r_{s} + r_{s-1} + \cdots + r_{s-i+1})}(\bF_{q}) : \\
\rk(X) = r_{s-i}
\end{array}\right\}\rt) q^{\sum_{i=1}^{s}(k-1-e_{i})(r_{i}^{2} + 2r_{i}\sum_{1 \leq j < i}r_{j})} \\
&= \lt(\prod_{i=0}^{s-1}\#\left\{\begin{array}{c}
X \in \M_{r-(r_{s} + r_{s-1} + \cdots + r_{s-i+1})}(\bF_{q}) : \\
\rk(X) = r_{s-i}
\end{array}\right\}\rt) q^{(k-1)r^{2} - \sum_{i=1}^{s}e_{i}r_{i}^{2}  + 2\sum_{1 \leq j < i \leq s}e_{i}r_{i}r_{j}} \\
&= \lt(\prod_{i=0}^{s-1}\#\left\{\begin{array}{c}
X \in \M_{r-(r_{s} + r_{s-1} + \cdots + r_{s-i+1})}(\bF_{q}) : \\
\rk(X) = r_{s-i}
\end{array}\right\}\rt) \frac{q^{(k-1)r^{2}}}{\prod_{1 \leq i, j \leq s}q^{\min(e_{i},e_{j})r_{i}r_{j}}}.
\end{align*}

\

Lemma \ref{rk2} implies that (noting that $r - (r_{s} + r_{s-1} + \cdots + r_{s-i+1}) = r_{1} + \cdots + r_{s-i}$)
\begin{align*}
\prod_{i=0}^{s-1}\#\left\{\begin{array}{c}
X \in \M_{r-(r_{s} + r_{s-1} + \cdots + r_{s-i+1})}(\bF_{q}) : \\
\rk(X) = r_{s-i}
\end{array}\right\} &= \prod_{i=0}^{s-1} q^{(r_{1} + \cdots + r_{s-i-1} + r_{s-i})^{2} - (r_{1} + \cdots + r_{s-i-1})^{2}}\frac{\prod_{i=r_{1} + \cdots + r_{s-i-1}+1}^{r_{1} + \cdots + r_{s-i-1} + r_{s-i}}(1 - q^{-i})^{2}}{\prod_{i=1}^{r_{s-i}}(1 - q^{-i})}\\
&= q^{r^{2}}\frac{\prod_{i=1}^{r}(1 - q^{-i})^{2}}{(\prod_{i=1}^{r_{1}}(1 - q^{-i})) \cdots (\prod_{i=1}^{r_{s}}(1 - q^{-i}))} \\
&= q^{r^{2}}\frac{\prod_{i=1}^{r}(1 - q^{-i})^{2}}{\prod_{i=1}^{s}q^{-r_{i}^{2}}|\GL_{r_{i}}(\bF_{q})|},
\end{align*}
where $r_{0} := 0$. Hence, continuing the previous computations, we have
\begin{align*}
\#\{pA \in p\M_{r}(R) : \cok_{R}(pA) \simeq G\} &= \lt(\prod_{i=0}^{s-1}\#\left\{\begin{array}{c}
X \in \M_{r-(r_{s} + r_{s-1} + \cdots + r_{s-i+1})}(\bF_{q}) : \\
\rk(X) = r_{s-i}
\end{array}\right\}\rt) \frac{q^{(k-1)r^{2}}}{\prod_{1 \leq i, j \leq s}q^{\min(e_{i},e_{j})r_{i}r_{j}}} \\
&= q^{kr^{2}}\frac{\prod_{i=1}^{r}(1 - q^{-i})^{2}}{\prod_{i=1}^{s}q^{-r_{i}^{2}}|\GL_{r_{i}}(\bF_{q})|(\prod_{1 \leq i, j \leq s}q^{\min(e_{i},e_{j})r_{i}r_{j}})} \\
&= q^{kr^{2}}\frac{\prod_{i=1}^{r}(1 - q^{-i})^{2}}{|\Aut_{R}(G)|},
\end{align*}
where we used Lemma \ref{Aut}. This finishes the proof for the first case where $e_{1} < k$ (i.e., $p^{k-1}G = 0$) since we have $u_{q}(G) = 0$.

\

\textbf{Case 2}. The second case is when $e_{1} = k$ (i.e., $p^{k-1}G \neq 0$). Note that to guarantee $\cok(pA) \simeq G$, we must satisfy the following criteria:

\bi
	\item We must have $B''_{e_{2}+1} = B''_{e_{2}+2} = \cdots = B''_{k-1} = 0$ to ensure that the product decomposition of $\cok(pA)$ does not have any copies of the form $R/p^{e}R$ with $e_{2}< e < e_{1} = k$. 
	\item Unlike Case 1, there is no requirement that $\rk(B''_{e_{1}}) = r_{1}$ because $e_{1} = k$ (and there is no $B''_{k}$).
\ei 
Up to here, we counted all the possibilities for $A_{1}, \dots, A_{k-1}$. Note that $q^{(k-1 - e_{1})r_{1}^{2}}$ is not multiplied unlike in Case 1. Following similar steps to Case 1, we have

\begin{align*}
\#\{pA \in p\M_{r}(R) : \cok_{R}(pA) \simeq G\} &= \frac{\lt(\prod_{i=0}^{s-1}\#\left\{\begin{array}{c}
X \in \M_{r-(r_{s} + r_{s-1} + \cdots + r_{s-i+1})}(\bF_{q}) : \\
\rk(X) = r_{s-i}
\end{array}\right\}\rt)}{q^{(k-1 - e_{1})r_{1}^{2}} \#\left\{\begin{array}{c}
X \in \M_{r_{1}}(\bF_{q}) : \\
\rk(X) = r_{1}
\end{array}\right\} } \frac{q^{(k-1)r^{2}}}{\prod_{1 \leq i, j \leq s}q^{\min(e_{i},e_{j})r_{i}r_{j}}} \\
&= \frac{1}{q^{(k-1 - e_{1})r_{1}^{2}} \#\left\{\begin{array}{c}
X \in \M_{r_{1}}(\bF_{q}) : \\
\rk(X) = r_{1}
\end{array}\right\} }  \frac{q^{kr^{2}}\prod_{i=1}^{r}(1 - q^{-i})^{2}}{|\Aut_{R}(G)|} \\
&= \frac{q^{kr^{2}}\prod_{i=1}^{r}(1 - q^{-i})^{2}}{|\Aut_{R}(G)|\prod_{i=1}^{r_{1}}(1 - q^{-i})}
\end{align*}
because $e_{1} = k$ and
\[\#\left\{\begin{array}{c}
X \in \M_{r_{1}}(\bF_{q}) : \\
\rk(X) = r_{1}
\end{array}\right\} = |\GL_{r_{1}}(\bF_{q})| = q^{r_{1}^{2}} \prod_{i=1}^{r_{1}}(1 - q^{-i}).\]
This finishes the proof.
\end{proof}

\

\begin{rmk} We note that when $p^{k-1}G = 0$ (i.e., $u_{q}(G) = 0$), the conclusion of Lemma \ref{R} can be deduced from the arguments by Friedman and Washington in \cite{FW}. However, when $p^{k-1}G \neq 0$ (i.e., $u_{q}(G) > 0$), Lemma \ref{R} is a new result.
\end{rmk}

\


\subsection{Proof of Theorem \ref{main}} To show Theorem \ref{main}, it is enough to show the following lemma:

\begin{lem}\label{final} Given any $pM_{1}, \dots, pM_{d-1} \in p\M_{n}(\bZ/p^{k}\bZ)$, we have

\[\#\left\{\begin{array}{c}
Y \in \M_{n}(\bZ/p^{k}\bZ) : \\
\cok_{R}(Y + \bar{t}(-I_{n} + pM_{1}) + \bar{t}^{2}pM_{2} + \cdots + \bar{t}^{d-1}pM_{d-1}) \simeq G, \\
Y \equiv \bar{X} \pmod{p}
\end{array}\right\} = \#\left\{\begin{array}{c}
X \in \M_{n}(\bZ/p^{k}\bZ) : \\
\cok_{R}(X - \bar{t}I_{n}) \simeq G, \\
X \equiv \bar{X} \pmod{p}
\end{array}\right\}.\]
\end{lem}

\

\begin{proof}[Proof that Lemma \ref{final} implies Theorem \ref{main}] Once we have Lemma \ref{final}, we can use Lemma \ref{R} so that

\begin{align*}
q^{(k-1)n^{2}}&\frac{q^{r_{q}(G)^{2}}\prod_{i=1}^{r_{q}(G)}(1 - q^{-i})^{2}}{|\Aut_{R}(G)|\prod_{i=1}^{u_{q}(G)}(1 - q^{-i})} \\
&= \#\left\{\begin{array}{c}
Z \in \M_{n}(R) : \\
\cok_{R}(Z) \simeq G \\
\text{and } Z \equiv \bar{X} - \bar{t}I_{n} \pmod{p}
\end{array}\right\}  \\
&= \sum_{pM_{1}, pM_{2}, \dots, pM_{d-1} \in p\M_{n}(\bZ/p^{k}\bZ)} \#\left\{\begin{array}{c}
Y \in \M_{n}(\bZ/p^{k}\bZ) : \\
\cok_{R}(Y + \bar{t}(-I_{n} + pM_{1}) + \bar{t}^{2}pM_{2} + \cdots + \bar{t}^{d-1}pM_{d-1}) \simeq G, \\
Y \equiv \bar{X} \pmod{p}
\end{array}\right\} \\
&= \sum_{pM_{1}, pM_{2}, \dots, pM_{d-1} \in p\M_{n}(\bZ/p^{k}\bZ)} \#\left\{\begin{array}{c}
X \in \M_{n}(\bZ/p^{k}\bZ) : \\
\cok_{R}(X - \bar{t}I_{n}) \simeq G, \\
X \equiv \bar{X} \pmod{p}
\end{array}\right\} \\
&= p^{(d-1)(k-1)n^{2}} \#\left\{\begin{array}{c}
X \in \M_{n}(\bZ/p^{k}\bZ) : \\
\cok_{R}(X - \bar{t}I_{n}) \simeq G, \\
X \equiv \bar{X} \pmod{p}
\end{array}\right\} \\
&= p^{(d-1)(k-1)n^{2}} \#\left\{\begin{array}{c}
X \in \M_{n}(\bZ/p^{k}\bZ) : \\
\cok(P(X)) \simeq G, \\
X \equiv \bar{X} \pmod{p}
\end{array}\right\},
\end{align*}

\

where we also used Lemma \ref{Lee}. Since $q = p^{d}$, this implies (\ref{goal}), which proves Theorem \ref{main}. 
\end{proof}

\

\begin{rmk}\label{special} Note that when $d = 1$, we have $R = \bZ/p^{k}\bZ$ and Lemma \ref{final} easily follows for this case. When $d = 2$, we have a bijection
\[\left\{\begin{array}{c}
Y \in \M_{n}(\bZ/p^{k}\bZ) : \\
\cok_{R}(Y + \bar{t}(-I_{n} + pM_{1})) \simeq G, \\
Y \equiv \bar{X} \pmod{p}
\end{array}\right\} \ra \left\{\begin{array}{c}
X \in \M_{n}(\bZ/p^{k}\bZ) : \\
\cok_{R}(X - \bar{t}I_{n}) \simeq G, \\
X \equiv \bar{X} \pmod{p}
\end{array}\right\}\]

given by $Y \mapsto X = Y(I_{n} - pM_{1})^{-1} = Y(I_{n} + pM_{1} + \cdots + p^{k-1}M_{1}^{k-1})$, which appeared in \cite[Section 3.2]{CK}. However, starting from $d = 3$, it is more complicated to show Lemma \ref{final}, and we resolve this complication for the rest of this paper. Before we get into the general proof, we consider the case $d = 3$ with $k = 2$ to explain some part of our argument more clearly. In this case, our goal is to construct a bijection
\[\left\{\begin{array}{c}
Y \in \M_{n}(\bZ/p^{2}\bZ) : \\
\cok_{R}(Y + \bar{t}(-I_{n} + pM_{1}) + \bar{t}^{2}pM_{2}) \simeq G, \\
Y \equiv \bar{X} \pmod{p}
\end{array}\right\} \ra \left\{\begin{array}{c}
X \in \M_{n}(\bZ/p^{2}\bZ) : \\
\cok_{R}(X - \bar{t}I_{n}) \simeq G, \\
X \equiv \bar{X} \pmod{p}
\end{array}\right\}.\]

Consider the map $Y \mapsto X = Y(I_{n} + pM_{1})(I_{n} + pYM_{2})$. This is a bijection because we are fixing $M_{1}$ and $M_{2}$, whose entries consist of elements of $\{0, 1, \dots, p-1\}$ modulo $p^{2}$, and $pY$ only depends on $\bar{X}$. We need to show that
\[\cok_{R}(X - \bar{t}I_{n}) \simeq G\]

when $X = Y(I_{n} + pM_{1})(I_{n} + pYM_{2}) - \bar{t}I_{n}$. Since $p^{2} = 0$ in $\bZ/p^{2}\bZ$, we have
\begin{align*}
(X- \bar{t}I_{n})(I_{n} - pYM_{2})(I_{n} - pM_{1}) &= (Y(I_{n} + pM_{1})(I_{n} + pYM_{2}) - \bar{t}I_{n})(I_{n} - pYM_{2})(I_{n} - pM_{1}) \\
&= Y - \bar{t}(I_{n} - pYM_{2})(I_{n} - pM_{1}) \\
&= Y + \bar{t}(-I_{n} + pM_{1} + pYM_{2}) \\
&= Y(I_{n} + \bar{t}pM_{2}) + \bar{t}(-I_{n} + pM_{1})
\end{align*}

and
\begin{align*}
(Y(I_{n} + \bar{t}pM_{2}) + \bar{t}(-I_{n} + pM_{1}))(I_{n} - \bar{t}pM_{2}) &= Y + \bar{t}(-I_{n} + pM_{1})(I_{n} - \bar{t}pM_{2}) \\
&= Y + \bar{t}(-I_{n} + pM_{1}) + \bar{t}^{2}pM_{2}.
\end{align*}

Thus, since any sum of an invertible matrix and a multiple of $pI_{n}$ is invertible in $\M_{n}(R)$, we have
\begin{align*}
\cok_{R}(X - \bar{t}I_{n}) &= \cok_{R}(Y(I_{n} + pM_{1})(I_{n} + pYM_{2}) - \bar{t}I_{n}) \\
&\simeq \cok_{R}(Y + \bar{t}(-I_{n} + pM_{1}) + \bar{t}^{2}pM_{2}) \\
&\simeq G,
\end{align*}

as desired.

\

\hspace{3mm} For $k = 2$ with general $d \geq 1$, the map
\[\left\{\begin{array}{c}
Y \in \M_{n}(\bZ/p^{2}\bZ) : \\
\cok_{R}(Y + \bar{t}(-I_{n} + pM_{1}) + \bar{t}^{2}pM_{2} + \cdots + \bar{t}^{d-1}pM_{d-1}) \simeq G, \\
Y \equiv \bar{X} \pmod{p}
\end{array}\right\} \ra \left\{\begin{array}{c}
X \in \M_{n}(\bZ/p^{2}\bZ) : \\
\cok_{R}(X - \bar{t}I_{n}) \simeq G, \\
X \equiv \bar{X} \pmod{p}
\end{array}\right\}\]

given by $Y \mapsto X = Y(I_{n} + pY^{d-2}M_{d-1})(I_{n} + pY^{d-3}M_{d-2}) \cdots (I_{n} + pYM_{2})(I_{n} + pM_{1})$ turns out to be a well-defined bijection, showing (\ref{goal}) for this specific case. (Showing that this is well-defined is tantamount to the proof of Lemma \ref{final1}, which is a step in showing Lemma \ref{final}.) However, for $k \geq 2$, it is difficult to write down an explicit bijection unless $d = 2$. Hence, we use an indirect approach, using induction on $k$.
\end{rmk}

\

\subsection{Proof of Lemma \ref{final}} Recall that we fix a finite size $R$-module $G$, where $R = (\bZ/p^{k}\bZ)[t]/(P(t))$ for fixed $k \in \bZ_{\geq 1}$. In particular, we have $p^{k}G = 0$. The proof of Lemma \ref{final} is by induction on $k$, and the following lemma (eventually used as Corollary \ref{final4} and Remark \ref{crucial}) is crucial for the induction step:

\begin{lem}\label{final1} Fix $X' \in \M_{n}(\bZ/p^{k-1}\bZ)$ such that $\cok_{R/p^{k-1}R}(X' - \bar{t}I_{n}) \simeq G/p^{k-1}G$. Fix any $pM_{1}, \dots, pM_{d-1} \in p\M_{n}(\bZ/p^{k}\bZ)$ and $p^{k-1}M'_{1}, \dots, p^{k-1}M'_{d-1} \in p^{k-1}\M_{n}(\bZ/p^{k}\bZ)$. We have
\begin{align*}
&\#\left\{\begin{array}{c}
Y \in \M_{n}(\bZ/p^{k}\bZ) : \cok_{R}(Y + \bar{t}(-I_{n} + pM_{1} + p^{k-1}M'_{1}) + \sum_{j=2}^{d-1}\bar{t}^{j}(pM_{j} + p^{k-1}M'_{j}) \simeq G, \\
Y \equiv X' \pmod{p^{k-1}}
\end{array}\right\} \\
&= \#\left\{\begin{array}{c}
X \in \M_{n}(\bZ/p^{k}\bZ) : \cok_{R}(X + \bar{t}(-I_{n} + pM_{1}) + \sum_{j=2}^{d-1}\bar{t}^{j}pM_{j}) \simeq G, \\
X \equiv X' \pmod{p^{k-1}}
\end{array}\right\}
\end{align*}
\end{lem}

\begin{proof} Consider the map
\begin{align*}
&\left\{\begin{array}{c}
Y \in \M_{n}(\bZ/p^{k}\bZ) : \cok_{R}(Y + \bar{t}(-I_{n} + pM_{1} + p^{k-1}M'_{1}) + \sum_{j=2}^{d-1}\bar{t}^{j}(pM_{j} + p^{k-1}M'_{j}) \simeq G, \\
Y \equiv X' \pmod{p^{k-1}}\end{array}\right\} \\
&\ra \left\{\begin{array}{c}
X \in \M_{n}(\bZ/p^{k}\bZ) : \cok_{R}(X + \bar{t}(-I_{n} + pM_{1}) + \sum_{j=2}^{d-1}\bar{t}^{j}pM_{j}) \simeq G, \\
X \equiv X' \pmod{p^{k-1}}
\end{array}\right\}
\end{align*}

given by
\[Y \mapsto X = Y(I_{n} + p^{k-1}Y^{d-2}M'_{d-1}) (I_{n} + p^{k-1}Y^{d-3}M'_{d-2}) \cdots (I_{n} + p^{k-1}Y^{2}M'_{3}) (I_{n} + p^{k-1}YM'_{2}) (I_{n} + p^{k-1}M'_{1}).\]

Note that $X \equiv Y \equiv X' \pmod{p^{k-1}}$. To show that this map is well-defined, we need to show that
\[\cok_{R}(X  + \bar{t}(-I_{n} + pM_{1}) + \bar{t}^{2}pM_{2} + \cdots + \bar{t}^{d-1}pM_{d-1}) \simeq G.\]

Once this is shown, we can immediately conclude that the map is a bijection because we work over $\bZ/p^{k}\bZ$: for any $m \in \bZ_{\geq 0}$, the matrix $p^{k-1}Y^{m}$ only depends on the $m$-th power of the residue class of $Y$ modulo $p$, which only depends on $X'^{m}$. Given any $A, B \in \M_{n}(R)$, we shall write $A \sim B$ when $B = QAQ'$ for some $Q, Q' \in \GL_{n}(R)$. In particular, if $A \sim B$, then $\cok_{R}(A) \simeq \cok_{R}(B)$. Since we work over $\bZ/p^{k}\bZ$, we have $p^{k} = 0$, so
\begin{align*}
&X  + \bar{t}(-I_{n} + pM_{1}) + \bar{t}^{2}pM_{2} + \cdots + \bar{t}^{d-1}pM_{d-1} \\
&= Y(I_{n} + p^{k-1}Y^{d-2}M'_{d-1}) (I_{n} + p^{k-1}Y^{d-3}M'_{d-2}) \cdots  (I_{n} + p^{k-1}M'_{1}) + \bar{t}(-I_{n} + pM_{1}) + \bar{t}^{2}pM_{2} + \cdots + \bar{t}^{d-1}pM_{d-1} \\
&\sim Y + \bar{t}(-I_{n} + pM_{1})(I_{n} - p^{k-1}M'_{1})(I_{n} - p^{k-1}YM'_{2}) \cdots (I_{n} - p^{k-1}Y^{d-2}M'_{d-1}) + \bar{t}^{2}pM_{2} + \cdots + \bar{t}^{d-1}pM_{d-1} \\
&= Y +  \bar{t}(-I_{n} + pM_{1} + p^{k-1}M'_{1} +  p^{k-1}YM'_{2} + p^{k-1}Y^{2}M'_{3} + \cdots + p^{k-1}Y^{d-2}M'_{d-1}) + \bar{t}^{2}pM_{2} + \cdots + \bar{t}^{d-1}pM_{d-1} \\
&= Y(I_{n} + \bar{t}p^{k-1}M'_{2}) + \bar{t}(-I_{n} + pM_{1} + p^{k-1}M'_{1} + p^{k-1}Y^{2}M'_{3} + \cdots + p^{k-1}Y^{d-2}M'_{d-1}) + \bar{t}^{2}pM_{2} + \cdots + \bar{t}^{d-1}pM_{d-1} \\
&\sim Y + \bar{t}(-I_{n} + pM_{1} + p^{k-1}M'_{1} + p^{k-1}Y^{2}M'_{3} + \cdots + p^{k-1}Y^{d-2}M'_{d-1})(I_{n} - \bar{t}p^{k-1}M'_{2}) + \bar{t}^{2}pM_{2} + \bar{t}^{3}pM_{3} + \cdots + \bar{t}^{d-1}pM_{d-1} \\
&= Y + \bar{t}(-I_{n} + pM_{1} + p^{k-1}M'_{1} + p^{k-1}Y^{2}M'_{3} + \cdots + p^{k-1}Y^{d-2}M'_{d-1}) + \bar{t}^{2}(pM_{2} + p^{k-1}M'_{2}) + \bar{t}^{3}pM_{3} + \cdots + \bar{t}^{d-1}pM_{d-1},
\end{align*}
where
\bi
	\item the first $\sim$ is given by multiplying $(I_{n} - p^{k-1}M'_{1})(I_{n} - p^{k-1}YM'_{2}) \cdots (I_{n} - p^{k-1}Y^{d-2}M'_{d-1})$ on the right to the previous line, and
	\item the second $\sim$ is given by multiplying $(I_{n} - \bar{t}p^{k-1}M'_{2})$ on the right to the previous line.
\ei
Continuing this computation, we end up having
\begin{align*}
&X  + \bar{t}(-I_{n} + pM_{1}) + \bar{t}^{2}pM_{2} + \cdots + \bar{t}^{d-1}pM_{d-1} \\
&\sim Y + \bar{t}(-I_{n} + pM_{1} + p^{k-1}M'_{1}) + \bar{t}^{2}(pM_{2} + p^{k-1}M'_{2}) +  \bar{t}^{3}(pM_{3} + p^{k-1}M'_{3}) + \cdots + \bar{t}^{d-1}(pM_{d-1} + p^{k-1}M'_{d-1}).
\end{align*}

Thus, we have
\begin{align*}
&\cok_{R}(X  + \bar{t}(-I_{n} + pM_{1}) + \bar{t}^{2}pM_{2} + \cdots + \bar{t}^{d-1}pM_{d-1}) \\
&\simeq \cok_{R}(Y + \bar{t}(-I_{n} + pM_{1} + p^{k-1}M'_{1}) + \bar{t}^{2}(pM_{2} + p^{k-1}M'_{2}) +  \bar{t}^{3}(pM_{3} + p^{k-1}M'_{3}) + \cdots + \bar{t}^{d-1}(pM_{d-1} + p^{k-1}M'_{d-1})) \\
&\simeq G,
\end{align*}
as desired.
\end{proof}

\

\hspace{3mm} We use Lemma \ref{final1} in the form of the following statement. This lets us translate the problem of counting matrices in $\M_{n}(\bZ/p^{k}\bZ)$ to a problem of counting matrices in $\M_{n}(R)$.

\begin{cor}\label{final2} Fix $X' \in \M_{n}(\bZ/p^{k-1}\bZ)$ such that  $\cok_{R/p^{k-1}R}(X' - \bar{t}I_{n}) \simeq G/p^{k-1}G$. Given any $pM_{1}, \dots, pM_{d-1} \in p\M_{n}(\bZ/p^{k}\bZ)$, we have

\begin{align*}
&\#\left\{\begin{array}{c}
Z \in \M_{n}(R) : \cok_{R}(Z) \simeq G \text{ and} \\
Z \equiv X' + \bar{t}(-I_{n} + pM_{1}) + \sum_{j=2}^{d-1}\bar{t}^{j}pM_{j} \pmod{p^{k-1}}
\end{array}\right\} \\
&= p^{n^{2}(d-1)} \cdot \#\left\{\begin{array}{c}
X \in \M_{n}(\bZ/p^{k}\bZ) : \\
\cok_{R}(X + \bar{t}(-I_{n} + pM_{1}) + \sum_{j=2}^{d-2}\bar{t}^{j}pM_{j}) \simeq G \\
\text{and } X \equiv X' \pmod{p^{k-1}}
\end{array}\right\}.
\end{align*}
\end{cor}

\begin{proof} We have
\begin{align*}
&\#\left\{\begin{array}{c}
Z \in \M_{n}(R) : \cok_{R}(Z) \simeq G \text{ and} \\
Z \equiv X' + \bar{t}(-I_{n} + pM_{1}) + \sum_{j=2}^{d-1}\bar{t}^{j}pM_{j} \pmod{p^{k-1}}
\end{array}\right\} \\
&= \sum_{\substack{p^{k-1}M_{1}', \dots, p^{k-1}M_{d-1}' \\ \in p^{k-1}\M_{n}(\bZ/p^{k}\bZ)}}  \#\left\{\begin{array}{c}
X \in \M_{n}(\bZ/p^{k}\bZ) : \\
\cok_{R}(X + \bar{t}(-I_{n} + pM_{1} + p^{k-1}M_{1}') + \sum_{j=2}^{d-1}\bar{t}^{j}(pM_{j} + p^{k-1}M_{j}')) \simeq G \\
\text{and } X \equiv X' \pmod{p^{k-1}}
\end{array}\right\} \\
&= \sum_{\substack{p^{k-1}M_{1}', \dots, p^{k-1}M_{d-1}' \\ \in p^{k-1}\M_{n}(\bZ/p^{k}\bZ)}} \#\left\{\begin{array}{c}
X \in \M_{n}(\bZ/p^{k}\bZ) : \\
\cok_{R}(X + \bar{t}(-I_{n} + pM_{1}) + \sum_{j=2}^{d-1} \bar{t}^{j}pM_{j} \simeq G \\
\text{and } X \equiv X' \pmod{p^{k-1}}
\end{array}\right\},
\end{align*}

the last equality of which is given by applying Lemma \ref{final1}. Since the last summand does not depend on what we sum over, the result follows.
\end{proof}

\

\hspace{3mm} For our purpose, we need to manipulate Corollary \ref{final2}, which is already a manipulation of Lemma \ref{final1}, once more. In order to do so, we need the following lemma:

\begin{lem}\label{final3} Fix any $A, B \in \M_{n}(R/p^{k-1}R)$ such that $\cok_{R/p^{k-1}R}(A) \simeq G/p^{k-1}G \simeq \cok_{R/p^{k-1}R}(B)$. Then
\[\#\left\{\begin{array}{c}
Z \in \M_{n}(R) : \cok_{R}(Z) \simeq G \text{ and} \\
Z \equiv A \pmod{p^{k-1}}
\end{array}\right\} = \#\left\{\begin{array}{c}
Z \in \M_{n}(R) : \cok_{R}(Z) \simeq G \text{ and} \\
Z \equiv B \pmod{p^{k-1}}
\end{array}\right\}.\]
\end{lem}

\begin{proof} By the paragraph before Lemme \ref{R}, we have
\[G \simeq (R/p^{e_{1}}R)^{r_{1}} \times (R/p^{e_{2}}R)^{r_{2}} \times \cdots \times (R/p^{e_{s}}R)^{r_{s}}\]
for some integers $k \geq e_{1} > e_{2} > \cdots > e_{s} \geq 1$ and $r_{1}, \dots, r_{s} \geq 1$. (The case $s = 0$ corresponds to trivial $G$.) Since $\cok_{R/p^{k-1}R}(A) \simeq G/p^{k-1}G$, there must be $Q, Q' \in \M_{n}(R/p^{k-1}R)$ such that
\[QAQ' = \diag(p^{e_{1}}, \dots, p^{e_{1}}, \dots, p^{e_{s}}, \dots, p^{e_{s}})\]
over $R/p^{k-1}R$, where each $p^{e_{j}}$ appears $r_{j}$ times in the diagonal matrix. Fix any lifts $\tilde{Q}, \tilde{Q}' \in \M_{n}(R)$ of $Q, Q' \in \M_{n}(R/p^{k-1}R)$. Then the lifts of $A$ correspond to the lifts of $QAQ'$ via the map $Z \mapsto \tilde{Q}Z\tilde{Q}'$, so
\[\#\left\{\begin{array}{c}
Z \in \M_{n}(R) : \cok_{R}(Z) \simeq G \text{ and} \\
Z \equiv A \pmod{p^{k-1}}
\end{array}\right\} = \#\left\{\begin{array}{c}
Z \in \M_{n}(R) : \cok_{R}(Z) \simeq G \text{ and} \\
Z \equiv \diag(p^{e_{1}}, \dots, p^{e_{1}}, \dots, p^{e_{s}}, \dots, p^{e_{s}}) \pmod{p^{k-1}}
\end{array}\right\}.\]
Replacing the role of $A$ and $B$, we get the desired identity.
\end{proof}

\

\hspace{3mm} The following is a more flexible form of Lemma \ref{final1} we need for the proof of Lemma \ref{final}. This follows form Corollary \ref{final2}, which follows from Lemma \ref{final1}:

\begin{cor}\label{final4} Given any $pM_{1}, \dots, pM_{d-1} \in p\M_{n}(\bZ/p^{k}\bZ)$,  fix any $X', X'' \in \M_{n}(\bZ/p^{k-1}\bZ)$ such that 
\[\cok_{R/p^{k-1}R}(X' + \bar{t}(-I_{n} + p\bar{M}_{1}) + \bar{t}^{2}p\bar{M}_{2} + \cdots + \bar{t}^{d-1}p\bar{M}_{d-1}) \simeq G/p^{k-1}G \simeq \cok_{R/p^{k-1}R}(X'' - \bar{t}I_{n}),\]

where $p\bar{M}_{i} \in p\M_{n}(R/p^{k-1}R)$ is $pM_{i}$ modulo $p^{k-1}$. Then
\[\#\left\{\begin{array}{c}
X \in \M_{n}(\bZ/p^{k}\bZ) : \\
\cok_{R}(X + \bar{t}(-I_{n} + pM_{1}) + \bar{t}^{2}pM_{2} + \cdots + \bar{t}^{d-1}pM_{d-1}) \simeq G \\
\text{and } X \equiv X' \pmod{p^{k-1}}
\end{array}\right\} = \#\left\{\begin{array}{c}
X \in \M_{n}(\bZ/p^{k}\bZ) : \\
\cok_{R}(X - \bar{t}I_{n}) \simeq G \\
\text{and } X \equiv X'' \pmod{p^{k-1}}
\end{array}\right\}.\]
\end{cor}

\begin{proof} By Corollary \ref{final2}, the desired statement is equivalent to
\[\#\left\{\begin{array}{c}
Z \in \M_{n}(R) : \cok_{R}(Z) \simeq G, \\
Z \equiv X' + \bar{t}(-I_{n} + pM_{1}) + \bar{t}^{2}pM_{2} + \cdots + \bar{t}^{d-1}pM_{d-1} \pmod{p^{k-1}}
\end{array}\right\} = \#\left\{\begin{array}{c}
Z \in \M_{n}(R) : \cok_{R}(Z) \simeq G, \\
Z \equiv X'' - \bar{t}I_{n} \pmod{p^{k-1}}
\end{array}\right\},\]

which follows from Lemma \ref{final3} by taking $A = X' + \bar{t}(-I_{n} + p\bar{M}_{1}) + \bar{t}^{2}p\bar{M}_{2} + \cdots + \bar{t}^{d-1}p\bar{M}_{d-1}$ and $B = X'' - \bar{t}I_{n}$.
\end{proof}

\

\begin{rmk}\label{crucial} From now on, we shall write
\[c_{G,d,n,k} := \#\left\{\begin{array}{c}
X \in \M_{n}(\bZ/p^{k}\bZ) : \\
\cok_{R}(X + \bar{t}(-I_{n} + pM_{1}) + \bar{t}^{2}pM_{2} + \cdots + \bar{t}^{d-1}pM_{d-1}) \simeq G \\
\text{and } X \equiv X' \pmod{p^{k-1}}
\end{array}\right\}\]
because by Lemma \ref{final4}, this quantity does not depend on $X' \in \M_{n}(R/p^{k-1}R)$ nor $pM_{1}, \dots, pM_{d-1} \in p\M_{n}(R/p^{k-1}R)$ as long as
\[\cok_{R/p^{k-1}R}(X' + \bar{t}(-I_{n} + p\bar{M}_{1}) + \bar{t}^{2}p\bar{M}_{2} + \cdots + \bar{t}^{d-1}p\bar{M}_{d-1}) \simeq G/p^{k-1}G,\]
where $p\bar{M}_{i} \in p\M_{n}(R/p^{k-1}R)$ is $pM_{i}$ modulo $p^{k-1}$. Whenever we use the notation $c_{G,d,n,k}$ in the proof of Lemma \ref{final}, we are implicitly using Lemma \ref{final4}.
\end{rmk}

\

\hspace{3mm} We are now ready to prove Lemma \ref{final}:

\begin{proof}[Proof of Lemma \ref{final}] We proceed by induction on $k \geq 1$. Let $k=1$. Then the statement we consider trivially follows because on the left-hand side of the desired identity, all $pM_{i} = 0$ as we work over $\bZ/p\bZ$. Suppose that the statement holds over $\bZ/p\bZ, \dots, \bZ/p^{k-1}\bZ$, and we work over $\bZ/p^{k}\bZ$ to finish the induction. We have
\begin{align*}
\#\left\{\begin{array}{c}
X \in \M_{n}(\bZ/p^{k}\bZ) : \\
\cok_{R}(X - \bar{t}I_{n}) \simeq G \\
\text{and } X \equiv \bar{X} \pmod{p}
\end{array}\right\}  &= \sum_{\substack{X' \in \M_{n}(\bZ/p^{k-1}\bZ): \\ \cok_{R/p^{k-1}R}(X' - \bar{t}I_{n}) \simeq G/p^{k-1}G, \\ 
X' \equiv \bar{X} \pmod{p}}} \#\left\{\begin{array}{c}
X \in \M_{n}(\bZ/p^{k}\bZ) : \\
\cok_{R}(X - \bar{t}I_{n}) \simeq G \\
\text{and } X \equiv X' \pmod{p^{k-1}}
\end{array}\right\} \\
&= \sum_{\substack{X' \in \M_{n}(\bZ/p^{k-1}\bZ): \\ \cok_{R/p^{k-1}R}(X' - \bar{t}I_{n}) \simeq G/p^{k-1}G, \\ X' \equiv \bar{X} \pmod{p}}} c_{G,d,n,k} \\
&=  c_{G,d,n,k} \cdot \#\left\{\begin{array}{c}
X' \in \M_{n}(\bZ/p^{k-1}\bZ) : \\
\cok_{R/p^{k-1}R}(X' - \bar{t}I_{n}) \simeq G/p^{k-1}G \\
\text{and } X' \equiv \bar{X} \pmod{p}
\end{array}\right\}.
\end{align*}

We may now use the induction hypothesis on the last quantity to obtain

\begin{align*}
&\#\left\{\begin{array}{c}
X \in \M_{n}(\bZ/p^{k}\bZ) : \\
\cok_{R}(X - \bar{t}I_{n}) \simeq G \\
\text{and } X \equiv \bar{X} \pmod{p}
\end{array}\right\} \\
&=  c_{G,d,n,k} \cdot \#\left\{\begin{array}{c}
X' \in \M_{n}(\bZ/p^{k-1}\bZ) : \\
\cok_{R/p^{N}R}(X' - \bar{t}I_{n}) \simeq G/p^{k-1}G \\
\text{and } X' \equiv \bar{X} \pmod{p}
\end{array}\right\} \\
&=  c_{G,d,n,k} \cdot \#\left\{\begin{array}{c}
X' \in \M_{n}(\bZ/p^{k-1}\bZ) : \\
\cok_{R/p^{k-1}R}(X' + \bar{t}(-I_{n} + p\bar{M}_{1}) + \bar{t}^{2}p\bar{M}_{2} + \cdots + \bar{t}^{d-1}p\bar{M}_{d-1}) \simeq G/p^{k-1}G \\
\text{and } X' \equiv \bar{X} \pmod{p}
\end{array}\right\},
\end{align*}

where $p\bar{M}_{i}$ is $pM_{i}$ modulo $p^{k-1}$ in $\M_{n}(\bZ/p^{k-1}\bZ)$. 

\

\hspace{3mm} Continuing the computation, we have

\begin{align*}
&\#\left\{\begin{array}{c}
X \in \M_{n}(\bZ/p^{k}\bZ) : \\
\cok_{R}(X - \bar{t}I_{n}) \simeq G \\
\text{and } X \equiv \bar{X} \pmod{p}
\end{array}\right\} \\
&=  c_{G,d,n,k} \cdot \#\left\{\begin{array}{c}
X' \in \M_{n}(\bZ/p^{k-1}\bZ) : \\
\cok_{R/p^{k-1}R}(X' + \bar{t}(-I_{n} + p\bar{M}_{1}) + \bar{t}^{2}p\bar{M}_{2} + \cdots + \bar{t}^{d-1}p\bar{M}_{d-1}) \simeq G/p^{k-1}G \\
\text{and } X' \equiv \bar{X} \pmod{p}
\end{array}\right\}\\
&= \sum_{\substack{X' \in \M_{n}(\bZ/p^{k-1}\bZ): \\ \cok_{R}(X' + \bar{t}(-I_{n}+p\bar{M}_{1}) + \bar{t}^{2}p\bar{M}_{2} + \cdots + \bar{t}^{d-1}p\bar{M}_{d-1}) \simeq G/p^{k-1}G, \\ X \equiv \bar{X} \pmod{p}}} c_{G,d,n,k} \\
&= \sum_{\substack{X' \in \M_{n}(\bZ/p^{k-1}\bZ): \\ \cok_{R}(X' + \bar{t}(-I_{n}+p\bar{M}_{1}) + \bar{t}^{2}p\bar{M}_{2} + \cdots + \bar{t}^{d-1}p\bar{M}_{d-1}) \simeq G/p^{k-1}G, \\ X \equiv \bar{X} \pmod{p}}} \#\left\{\begin{array}{c}
X \in \M_{n}(\bZ/p^{k}\bZ) : \\
\cok_{R}(X + \bar{t}(-I_{n} + p M_{1}) + \bar{t}^{2}p M_{2} + \cdots + \bar{t}^{d-1}p M_{d-1}) \simeq G \\
\text{and } X \equiv X' \pmod{p^{k-1}}
\end{array}\right\} \\
&= \#\left\{\begin{array}{c}
X \in \M_{n}(\bZ/p^{k}\bZ): \\
\cok_{R}(X + \bar{t}(-I_{n} + pM_{1}) + \bar{t}^{2}pM_{2} +\cdots + \bar{t}^{d-1}pM_{d-1}) \simeq G \\
\text{and } X \equiv \bar{X} \pmod{p}
\end{array}\right\}.
\end{align*}

This finishes the proof.
\end{proof}


\end{document}